\documentclass{amsart}
\usepackage{amsmath, amssymb,epic,graphicx,mathrsfs,enumerate}
\usepackage[all]{xy}
\usepackage{color}
\usepackage{comment}

\usepackage{amsthm}
\usepackage{amssymb}
\usepackage{latexsym}
\usepackage{longtable}
\usepackage{epsfig}
\usepackage{amsmath}
\usepackage{hhline}

\DeclareMathOperator{\inn}{Inn} \DeclareMathOperator{\perm}{Sym}
 \DeclareMathOperator{\soc}{soc}
\DeclareMathOperator{\aut}{Aut} 
 
\DeclareMathOperator{\B}{B}

\DeclareMathOperator{\ran}{rank} \DeclareMathOperator{\frat}{Frat}
\DeclareMathOperator{\ssl}{SL}

  \DeclareMathOperator{\diam}{diam}

\DeclareMathOperator{\ra}{rank}

\DeclareMathOperator{\GL}{GL}

\DeclareMathOperator{\End}{End} 
\DeclareMathOperator{\der}{Der}

\newcommand{\remm}{\mathrm{m}}

\newcommand{\lst}{\Lambda\!^*\!(G)}
\newcommand{\lstt}{\Lambda_1\!^*\!(G)}
\newcommand{\lti}{\tilde\Lambda\!^*\!(G)}

\newcommand{\w}{\widetilde}

\newcommand{\C}{\mathbb C}

\newcommand{\rg}{\Gamma}
\newcommand{\N}{\mathbb N}

\renewcommand{\emptyset}{\varnothing}

\newtheorem{thm}{Theorem}[section]
\newtheorem{cor}[thm]{Corollary}
\newtheorem{lemma}[thm]{Lemma}
\newtheorem{prop}[thm]{Proposition} 
 \newtheorem{defn}[thm]{Definition}
\newtheorem{question}[]{Question} \newtheorem{con}[]{Conjecture}
\numberwithin{equation}{section}

\renewcommand{\footnote}{\endnote}
\newcommand{\ignore}[1]{}\makeglossary

\begin{document}
	\bibliographystyle{amsplain}
	\title[Graphs and generation]{Graphs encoding the generating properties \\of a finite group}
	
\author[Cristina Acciarri]{Cristina Acciarri}
\address{Cristina Acciarri\\ Department of Mathematics, University of Brasilia, 70910-900 Bras\'ilia DF, Brazil\\ 
email:acciarricristina@yahoo.it}
\author[Andrea Lucchini]{Andrea Lucchini}
\address{Andrea Lucchini\\ Universit\`a degli Studi di Padova\\  Dipartimento di Matematica \lq\lq Tullio Levi-Civita\rq\rq\\ Via Trieste 63, 35121 Padova, Italy\\email: lucchini@math.unipd.it}

\thanks{Research partially supported by MIUR-Italy via PRIN \lq Group theory and applications\rq. The first author is also supported by CAPES-Brazil (bolsista da Capes/Pesquisa P\'os-doutoral no Exterior/Processo n. 88881.119002/2016-01) and thanks the Universit\`a degli Studi di Padova for support and hospitality that she enjoyed during her visit to Padua.}

	\begin{abstract} 
Assume that $G$ is a finite group. For every $a, b \in\mathbb N,$ we define a  graph $\Gamma_{a,b}(G)$ whose vertices correspond to the elements of $G^a\cup G^b$ and in which two tuples $(x_1,\dots,x_a)$ and $(y_1,\dots,y_b)$ are adjacent if and only if $\langle x_1,\dots,x_a,y_1,\dots,y_b \rangle =G.$  We study several properties of these graphs (isolated vertices, loops, connectivity, diameter of the connected components) and we investigate the relations between their properties and the group structure, with the aim of understanding which  information about $G$ is encoded by these graphs.

			\end{abstract}
	\maketitle

\section{Introduction}

The generating  graph $\Gamma(G)$ of a finite group $G$ is the graph defined on the elements of $G$ in such a way that two distinct vertices are connected by an edge if and only if they generate $G.$  It was defined by Liebeck and
Shalev in \cite{LS}, and has been further investigated by
many authors: see for example \cite{ bglmn,  beghm, bucr, ccl, CLis, GK, LM2, lm, LMRD, atsim} for some
of the range of questions that have been considered. Many deep
structural results about finite groups can be expressed in terms of
the generating graph, but of course  $\Gamma(G)$ encodes significant information only when $G$ is a 2-generator group. The aim of this paper is to introduce and investigate a wider family of graphs which encode the generating property of $G$ when $G$ is an arbitrary finite group.

\

We introduce the following definition. Assume that $G$ is a finite group and let $a$ and $b$ be non-negative integers. We define an undirected  graph $\Gamma_{a,b}(G)$ whose vertices correspond to the elements of $G^a\cup G^b$ and in which two tuples $(x_1,\dots,x_a)$ and $(y_1,\dots,y_b)$ are adjacent if and only $\langle x_1,\dots,x_a,y_1,\dots,y_b \rangle =G.$  
 Notice that $\Gamma_{1,1}(G)$ is the generating graph of $G,$ 	so these graphs can be viewed as a natural generalization of the generating graph. 

\

There may be many isolated vertices in the generating graph $\Gamma(G)$ of a finite group $G$ (for example if $N$ is a normal subgroup of $G$ and $G/N$ is not cyclic, then all the elements of $N$ correspond to isolated vertices). However,  \cite{CLis}  considers the subgraph $\Gamma^*(G)$
of $\Gamma(G)$ that is induced by all of the vertices that are not isolated and it is  proved that if $G$ is a 2-generator soluble group, then $\Gamma^*(G)$ is connected. This result is  equivalent to saying the \lq\lq swap conjecture\rq\rq \ is satisfied by the 2-generator finite soluble groups. Recall that the swap conjecture concerns the connectivity of the graph $\Sigma_d(G)$ in which the vertices are the ordered generating $d$-tuples and  two vertices $(x_1,\dots,x_d)$ and $(y_1,\dots,y_d)$ are adjacent if and only if they differ only by one entry.
Tennant and Turner \cite{TT} conjectured that the swap graph is connected for every group. Roman'kov \cite{rom} proved that the free metabelian group of rank 3 does not satisfy this conjecture but no counterexample is known
in the class of finite groups. There is a strong relation between the properties  of the swap graph  $\Sigma_{a+b}(G)$ and those of the graph $\Gamma_{a,b}^*(G),$ obtained from $\Gamma_{a,b}(G)$ by deleting the isolated vertices.  In particular we prove that if $\Sigma_{a+b}(G)$ is connected, then  $\Gamma_{a,b}^*(G)$  is also connected (see Lemma \ref{redsw}). Recently \cite{swap, disum} it has been proved that $\Sigma_d(G)$ is connected if either $d>d(G)$ or $d=d(G)$ and $G$ is soluble (where $d(G)$ is  the minimum number of generators of $G$). This can be used to prove the connectivity of  $\Gamma_{a,b}^*(G)$ in many cases: {\sl{the graphs $\Gamma_{a,b}^*(G)$ are connected, except possibly when $a+b=d(G)$ and $G$ is not soluble}} (see Corollary \ref{c15}).

\

Once is known that the graphs $\Gamma_{a,b}^*(G)$ are connected in most cases, the next step is to investigate their diameters. When $G$ is  soluble and 2-generated,  it has been recently proved \cite{diam} that the graph $\Gamma^*(G)$ has diameter at most 3: this bound is best possible, but it can be improved to 2 if $G$ satisfies the following additional property: $|\End_G(V)|>2$ for every non-trivial irreducible $G$-module $V$ which is $G$-isomorphic to a complemented chief factor of $G$ (which is true for example if the derived subgroup of $G$ is nilpotent or has odd order). In this paper we prove a more general result (see Theorem \ref{abd}): {\sl{assume that $G$ is a finite soluble group and that $(x_1,\dots,x_b)$ and $(y_1,\dots,y_b)$ are non-isolated vertices of $\Gamma_{a,b}(G)$: if either $a\neq 1$ or  $|\End_G(V)|>2$ for every non-trivial irreducible $G$-module $V$ which is $G$-isomorphic to a complemented chief factor of $G$, then there exists $(z_1,\dots,z_a)\in G^a$ such that $G=\langle z_1,\dots,z_a,x_1,\dots,x_b \rangle =
\langle z_1,\dots,z_a, y_1,\dots,y_b  \rangle.$}} We will give an example showing that when $a=1$ the previous statement does not remain true if we drop  the assumption on the order of the endomorphism group of the complemented chief factors. But in any case the previous result allows us to conclude that  {\sl{$\diam(\Gamma^*_{a,b}(G))\leq 4$ whenever $G$ is soluble and $a+b\geq d(G)$}} (see Corollary \ref{diametro}). These results lead also to a better understanding of the swap graph. For example we deduce that {\slshape{if $G$ is soluble and $|\End_G(V)|>2$ for every non-trivial irreducible $G$-module  $V$  which is $G$-isomorphic to a complemented chief factor of $G$, then the diameter of the swap graph $\Sigma_d(G)$ is at most $2d-1$}} (see Theorem \ref{swapdiam}).

\

The bound $\diam(\Gamma^*_{a,b}(G))\leq 4$ that we prove for finite soluble groups cannot be generalized to an arbitary finite group. Assume that $S$ is a finite non-abelian simple group and, for $d\geq 2,$ let $\tau_d(S)$  be the largest positive integer $r$ such that $S^{r}$ can be
generated by $d$ elements. In Section \ref{powsim} we will prove that if $a$ and $b$ are positive integers, then $\Gamma^*_{a,b}(S^{\tau_{a+b}(S)})$ is connected, however
$$\lim_{p\to \infty} \diam(\Gamma^*_{a,b}(\ssl(2,2^p)^{\tau_{a+b}(\ssl(2,2^p))}))=\infty.$$

\

In Section \ref{recog} we investigate how one can deduce information on $G$ from the knowledge of the graphs $\Gamma^*_{a,b}(G)$ for all the possible choices of $a$ and $b.$  More precisely we will denote by $\Lambda^*(G)$ the collection of all the connected components  of the graphs $\Gamma^*_{a,b}(G)$, for all the possible choices of $a, b$ in $\mathbb N$. However for each of the graphs in this family, we do not assume to know from which choice of $a, b$ it arises. Roughly speaking, we can think that we  packaged all the graphs $\Gamma^*_{a,b}(G)$ in a (quite spacious) box but that we did not pay enough attention during this operation and we lost  the information to which group $G$ these graphs correspond and the labels $a,b$: do not panic, a big amount  of the lost information can be reconstructed! We prove that {\slshape{from the knowledge of $\Lambda^*(G)$ we may
recover $d(G)$, $|G|$ and the labels  $a,b$, at least when $a+b>d(G)$}} (see Propositions \ref{ord}, \ref{dord} and \ref{detab}).
 Moreover considerations on the number of edges of the graphs in $\lst$ allows us to determine, for every $t\in \mathbb N,$ the number $\phi_G(t)$ of the ordered generating $t$-tuples of $G.$ Philip Hall \cite{hall} observed that the probability $\phi_G(t)/|G|^t$ of generating a given finite group $G$ by a random $t$-tuple of elements is given by
$$P_{G}(t)=\sum _{n \in
	\N}\frac{a_n(G)}{n^t}$$
where
$a_n(G)=\sum_{|G:H|=n}\mu _G(H)$ and $\mu$ is the M\"{o}bius function on the subgroup lattice of $G$.
In other words, for a given finite group $G,$ there exists a uniquely determined Dirichlet polynomial $P_G(s)$ (where $s$ is a complex variable) with the property that for $t\in\mathbb N$ the number $P_G(t)$ coincides with the probability of generating $G$ by $t$ randomly chosen elements.
The reciprocal of $P_G(s)$ is the \lq\lq probabilistic zeta function\rq\rq  \ of $G,$ studied by N. Boston \cite{boston}, A. Mann \cite{PFG} and the second author \cite{arac}. We prove that $P_G(s)$ can be determined from $\lst$ (see Theorem \ref{pzg}) and consequently we may also recover from $\lst$ all the information that can be determined from $P_G(s),$ taking advantages from a series of available results in the literature, about the relation between the arithmetic properties of the Dirichlet series $P_G(s)$ and the structure of $G$. In particular we  may deduce whether $G$ is soluble or supersoluble  and, for every prime power $n,$ determine the number of maximal subgroups of $G$ of index $n$. But we also prove that {\slshape{from $\lst$ we may deduce whether $G$ is nilpotent and the order of the Frattini subgroup}} (information that cannot be recovered from $P_G(s)).$ A possible development of this investigation could be to minimize the number of graphs in $\Lambda^*(G)$ that have to be considered in order to obtain information about $G$. From this point of view, we notice  that all the above mentioned properties of $G$ could be deduced taking into account only the graphs  of the form $\Gamma^*_{1,b}(G)$ for $b\in \mathbb N.$

\

The graphs $\Gamma_{1,b}(G)$ play also a central role in the last section of the paper. In \cite{ccl} an equivalence relation
$\equiv_{\remm}$ has been introduced, where two elements
are equivalent if each can be substituted for the other in any
generating set for $G$. This relation can be refined to 
a new sequence $\equiv_\remm^{(r)}$ of equivalence relations by saying that 
$x \equiv_\remm^{(r)}y$  if each can be substituted for the other in any
$r$-element generating set.
The relations $\equiv_\remm^{(r)}$ become finer as $r$ increases, and in \cite{ccl} the authors study
the value $\psi(G)$ of $r$ at which they stabilise to $\equiv_{\remm}$. Indeed results about $\equiv_{\remm}$, $\equiv_\remm^{(r)}$ and $\psi(G)$ can be reformulated and reinterpreted in terms of properties of the graphs  $\Gamma_{1,b}(G).$
 A significant role in this investigation is played by the groups $G$ with the property that $(g)$ is not isolated in the graph $\Gamma_{1,d(G)-1}(G)$ for every $g\neq 1$ (generalising a terminology  used for 2-generator groups, we say that $G$ has non-zero spread if it satisfies such property). 
In \cite{bgk}, Breuer, Guralnick and Kantor make the following remarkable conjecture: a 2-generated finite group has non-zero spread if and only if every proper quotient is cyclic. This conjecture has been recently proved by Burness, Guralnick and Scott \cite{bgh}. In the final part of the  paper we generalize this result, proving that  a finite group $G$ has non-zero spread if and only if  $d(G/N)<d(G)$ for every non-trivial normal subgroup $N$ of $G.$ (see Proposition \ref{ultimo}).

\section{The graphs $\Gamma_{a,b}(G)$ and $\Gamma^*_{a,b}(G)$.}

In this section we give the definition of the graphs $\Gamma_{a,b}(G)$ and $\Gamma^*_{a,b}(G)$ associated to a finite group $G$ and a pair $(a,b)$ of non-negative integers. Firstly we explore some properties of these graphs that follow easily from their definitions and then we investigate their connection with the so called \lq swap graph'. In particular we use this connection in order to deduce results about the connectivity of $\Gamma_{a,b}(G)$ and $\Gamma^*_{a,b}(G).$

\

Let $G$ be a finite group. We will denote by $d(G)$ the smallest cardinality of a generating set of $G.$ Moreover, 
given $d\in \mathbb N$, we will denote by $\Phi_G(d)$ the set of the ordered generating $d$-tuples of $G$ and by $\phi_G(d)$ the cardinality of this set.

\begin{defn}
Assume that $G$ is a finite group and let $a$ and $b$ be non-negative integers with $a\leq b$. We define an undirected  graph $\Gamma_{a,b}(G)$ whose vertices correspond to the elements of $G^a\cup G^b$ and in which two tuples $(x_1,\dots,x_a)$ and $(y_1,\dots,y_b)$ are adjacent if and only $\langle x_1,\dots,x_a,y_1,\dots,y_b \rangle =G.$ 
\end{defn}
 Clearly if $a+b<d(G),$ then $\Gamma_{a,b}(G)$ is an empty graph, so in general we will implicitly assume $a+b\geq d(G).$

\begin{defn}
$\Gamma_{a,b}^*(G)$ is the  graph obtained from $\Gamma_{a,b}(G)$ by deleting the isolated vertices.
\end{defn}

In the particular case when $a=0,$ the graph $\Gamma_{0,b}^*(G)$ is a star with one internal node, corresponding  to the 0-tuple, and $\phi_G(b)$ leaves, corresponding to the ordered generating $b$-tuples of $G.$ Notice that if $a\geq d(G)$, then $\Gamma_{a,a}(G)$ contains loops: if $G=\langle g_1,\dots,g_a\rangle$ then we have a loop around the vertex $(g_1,\dots,g_a).$

\

Let $d=a+b.$ If $a\neq b$ then $\Gamma_{a,b}(G)$ and $\Gamma^*_{a,b}(G)$ are bipartite graphs with two parts, one corresponding to the elements of $G^a$ and the other to the elements of $G^b$. We will use the notations $V_a$ and $V_b$ for the vertices of $\Gamma^*_{a,b}(G)$ corresponding, respectively, to elements of $G^a$ and $G^b.$ In particular $\Gamma_{a,b}(G)$ has $|G|^a+|G|^b$ vertices and there exists a bijective correspondence between  $\Phi_G(d)$ and the set of the edges of $\Gamma_{a,b}(G)$: indeed if $\langle g_1,\dots,g_d\rangle=G,$ then $(g_1,\dots,g_a)$ and $(g_{a+1},\dots,g_{d})$ are adjacent  vertices of the graph. Hence the number of edges of $\Gamma_{a,b}(G)$ (which coincides with the number of edges of $\Gamma^*_{a,b}(G)$) is  $\phi_G(d).$ The situation is different if $a=b.$ In that case $\Gamma_{a,a}(G)$ has $|G|^a$ vertices,
$\phi_G(a)$ loops and other $(\phi_G(d)-\phi_G(a))/2$ edges connecting two different vertices (in other words if $e$ is the the number of edges, excluding the loops, and $l$ is the number of loops, then $2e+l=\phi_G(d)$);
 indeed the two elements $(g_1,\dots,g_a,g_{a+1},\dots,g_d)$ and 
$(g_{a+1},\dots,g_d,g_1,\dots,g_a)$ give rise to the same edge in $\Gamma_{a,a}(G)$.

\begin{lemma}\label{tri} Let  $G$ be any  non-trivial finite group and let $a$ be any positive integer.  
	Then  any edge, which is not a loop, of the  graph $\Gamma_{a,
		a}^*(G)$ lies in  a 3-cycle, except when $a=1$ and $G\cong C_2.$
\end{lemma}
\begin{proof} Take  any edge in  $\Gamma_{a,a}^*(G)$, which is not a loop, and let us call $x=(x_1,\ldots,x_a)$ and $y=(y_1,\ldots,y_a)$ its vertices.  If $x$ and $y$ are different from the tuple $(1,\ldots,1)$, then both vertices are adjacent to a third vertex  $z=(x_1y_1,\ldots,x_ay_a)$ and we are done. Next assume that one vertex, let us say $y$, has all trivial entries. This implies  that  $x$ is a generating $a$-tuple for $G$, so the vertex $x$ is adjacent to all other  vertices of  $\Gamma_{a,a}^*(G)$.  If $(a,G)\neq (1,C_2)$, then there exists a generating $a$-tuple for $G$ different from $x$, and this  is adjacent  to both $x$ and $y$. This concludes the proof.
\end{proof}

From the previous lemma it follows that no  connected component of $\Gamma_{a,a}^*(G)$ is  bipartite since a graph is bipartite if and only if it contains no odd cycles. Observe that if $G=1$, then, for every $a\in \mathbb N,$  the graph $\Gamma_{a,a}^*(G)$ consists of a unique vertex with a loop, so it is not bipartite either. In the case where $G$ is isomorphic to $C_2$ and $a=1$, the graph  $\Gamma_{1,1}^*(G)$ is again not bipartite since we have a loop on the vertex corresponding to the unique generator of $G$. 

\begin{lemma}\label{degone}If $|G|\geq 3$, then $\Gamma^*_{a,b}(G)$ contains a vertex $x$ of degree 1 if and only if $a=0,$ $b\geq d(G)$ and $x$ is one of
the $\phi_G(b)$ leaves of the star $\Gamma_{0,b}^*(G)\cong K_{1,\phi_G(b)  }$.	
\end{lemma} 

\begin{proof}
Assume that $x$ is a vertex of degree 1 in $\Gamma_{a,b}^*(G)$ and that $a>0.$ We may assume  $x=(x_1,\ldots,x_r)$ with $r\in \{a,b\}.$ Let $s=a+b-r.$ Then there exists $(y_1,\ldots,y_s)$ such that  $G=\langle x_1,\ldots,x_r, y_1,\ldots,y_s \rangle$. If $x_i\neq 1$ for some $i\in\{1,\dots,r\},$ then   $x$ is also adjacent to the tuple $(x_iy_1,y_2,\ldots,y_s)$, a contradiction. So $x=(1,\dots,1)$ and consequently $y=(y_1,\dots,y_s)$ is a tuple of generators for $G$. For every $\pi\in \perm(s),$ the element $y_\pi=(y_{1\pi},\dots,y_{s\pi})$ is adjacent to $x$. Since $x$ has degree 1, we must have   $y_1=\dots=y_s,$ $G=\langle y_1\rangle$ and $y_1$ is the unique element generating $G$: this implies 
$|G|\leq 2$.
\end{proof}

The M\"{o}bius function $\mu_G$  is the function defined on the lattice of subgroups of  $G$ by  $\sum_{K\geq H}\mu_G(K)=\delta_{H,G},$ where  $\delta_{G,G}=1$ and $\delta_{H,G}=0$ if $H\neq G$. The following is a consequence of \cite[Section 3]{xdir}.

\begin{lemma}  Let $a$ and $b$ be non-negative integers. Let  $x=(x_1,\ldots,x_r)\in G^r$ with $r\in \{a,b\}$ and set $K=\langle x_1,\ldots,x_r\rangle,$ $s=a+b-r$ and let $\delta_{a,b}(x)$ be the degree of $x$ in  $\Gamma_{a,b}(G)$. We have $$\delta_{a,b}(x)=\sum_{K\leq H}\mu_G(H)|H|^s.$$
In particular $|K|^s$	divides the degree $\delta_{a,b}(x)$ of $x$ in $\Gamma_{a,b}(G)$.
\end{lemma}

Recall that  for a $d$-generator finite group $G,$ the swap graph $\Sigma_d(G)$ is the graph
in which the vertices are the ordered generating $d$-tuples
and in which two vertices $(x_1,\dots,x_d)$ and $(y_1,\dots,y_d)$ are adjacent if and only if  they differ only by one entry. 

\begin{lemma}\label{redsw}
	If $\Sigma_{a+b}(G)$ is connected, then $\Gamma_{a,b}^*(G)$ is connected.
\end{lemma}

\begin{proof}Let $d=a+b.$ We write any generating $d$-tuple $\omega$ in the form $\omega=(\alpha,\beta),$ with $\alpha \in G^a$ and $\beta\in G^b.$ Now let $\sigma, \sigma^*$ be two non-isolated vertices of $\Gamma_{a,b}^*(G)$: there exist two generating $d$-tuples $\omega=(\alpha,\beta)$ and $\omega^*=(\alpha^*,\beta^*)$ with $\sigma\in \{\alpha,\beta\}$ and
	$\sigma^*\in \{\alpha^*,\beta^*\}.$ Since $\Sigma_d(G)$ is connected, there exists a path in $\Sigma_d(G)$ joining $\omega$ to $\omega^*.$ In order to complete our proof, it suffices to prove that if $$\omega_1=(\alpha_1,\beta_1),\dots,\omega_u=(\alpha_u,\beta_u)$$
	is a path in $\Sigma_d(G)$, then the vertices $\alpha_1,\beta_1,\alpha_2,\beta_2,\dots, \alpha_u, \beta_u$ belong to the same connected component of  $\Gamma_{a,b}^*(G).$ We prove this claim by induction on $u.$ The sentence is clearly true when $u=1.$ Assume $u\geq 2.$ By induction $\alpha_2,\beta_2,\dots, \alpha_u, \beta_u$ belong to the same connected component of  $\Gamma_{a,b}^*(G);$ so it is enough to show that $\alpha_1,\beta_1,\alpha_2, \beta_2$ belong to the same connected component. Since $(\alpha_1,\beta_1)$ and $(\alpha_2,\beta_2)$
	differ for only one entry, either $\alpha_1=\alpha_2$ or $\beta_1=\beta_2.$ The graph  $\Gamma_{a,b}^*(G)$ contains the path
	$\beta_1,\alpha_1=\alpha_2,\beta_2$ in the first case and the path $\alpha_1,\beta_1=\beta_2,\alpha_2$ in the second case. 
\end{proof}

The swap conjecture  states that $\Sigma_d(G)$ is connected for every finite group $G$ and every $d\geq d(G)$. In \cite{swap} it was proved that this conjecture is true if $d>d(G)$, while in \cite{disum}
it is proved that it is true also when $d=d(G)$ and $G$ is soluble. So we have:

\begin{cor}\label{c15} If $G$ is a finite group and either $a+b > d(G)$ or $a+b=d(G)$ and $G$ is soluble, then
	$\Gamma_{a,b}^*(G)$ is connected.
\end{cor}

It remains an open problem to decide whether $\Gamma^*_{a,b}(G)$ is connected when $a+b=d(G)$ and $G$ is unsoluble. We conjecture that the answer is positive. However we think that proving results in this direction would be quite difficult and would require 
deep information about the generation properties of the finite almost simple groups.

\

We conclude this section, with the following result, that will be used later.

\begin{lemma}\label{norsbg}Let $N$ be a normal subgroup of a finite group $G$ and let $a$ and $b$ be non-negative integers and assume that $a+b\geq d(G).$ If $\Gamma^*_{a,b}(G)$ is connected, then $\Gamma^*_{a,b}(G/N)$ is connected too.
\end{lemma}

This lemma is an easy consequence of the following result due to  Gasch\"utz \cite{Ga}.

\begin{thm}
Let $G$ be any group that can be generated by $d$ elements and $N$ be any    finite normal subgroup of $G.$ Let $\eta: G \to \bar G = G/N$ be the
natural homomorphism given by $\eta: g \to \bar g = Ng$ for all $g \in G.$ Then for any generating
$d$-tuple $(y_1, y_2, \dots, y_d)$ of elements of $G/N$ there exist elements $x_1, x_2,\dots, x_d \in G$ such that
$\langle x_1, x_2,\dots, x_d\rangle = G$ and $\bar x_i = y_i$ for $1 \leq i \leq d.$
\end{thm}

\section{Bounding the diameter of $\Gamma^*_{a,b}(G)$ when $G$ is soluble}\label{boundsol}

In \cite{diam} it is proved that if $G$ is a 2-generator finite soluble group, then the graph $\Gamma^*_{1,1}(G)$  obtained from the generating graph by removing the isolated vertices has a very small diameter: indeed $\diam(\Gamma^*_{1,1}(G))\leq 3.$ Moreover $\diam(\Gamma^*_{1,1}(G))\leq 2$ if
$G$  has the property that $|\End_G(V)|>2$ for every non-trivial irreducible $G$-module $V$ which is $G$-isomorphic to a complemented chief factor of $G$. The aim of this section is to bound $\diam(\Gamma^*_{a,b}(G))$ for arbitrary values of $a$ and $b$ when $G$ is soluble.

\

Before  dealing with  the general case of a soluble group $G$, we need to collect  in the next four lemmas a series of results in linear algebra. Denote by $M_{r\times s}(F)$ the set of the $r\times s$ matrices with coefficients over the field $F.$

\begin{lemma}\label{comcom}{\cite[Lemma 3]{CLis}}
	Let $V$ be a finite dimensional vector space over the field $F$.
	If $W_1$ and $W_2$ are subspaces of $V$ with $\dim W_1=\dim W_2$,
	then $V$ contains a subspace $U$ such that $V=W_1\oplus U=W_2\oplus U.$
\end{lemma}

\begin{lemma}\label{lemmanil} Assume that $a$ and $b$ are non-negative integers. Let $V$ be a vector space of dimension $\delta$ over a finite field $F$ and let $x=(v_1,\dots,v_a)$ and $y=(w_1,\dots,w_a)$ be two elements of $V^a$ with
	$\dim_F \langle v_1,\dots,v_a\rangle\geq \delta-b$ and $\dim_F \langle w_1,\dots,w_a\rangle\geq \delta-b$. Then there exists $z=(z_1,\dots,z_b)\in V^b$ such that $\langle v_1,\dots,v_a,z_1,\dots z_b\rangle=\langle w_1,\dots, w_a,z_1,\dots z_b\rangle=V.$ 
\end{lemma}

\begin{proof}
	Let $U_1=\langle v_1,\dots,v_a\rangle,$
	$U_2=\langle w_1,\dots,w_a\rangle$  and $s=\min\{\dim_FU_1,\dim_FU_2\}.$ Clearly we may assume $s<\delta.$
	We prove our claim by induction on $s.$ If $s=0,$ then $b\geq \delta$ and it suffices to choose $z_1,\dots,z_b$ so that $\langle z_1,\dots,z_b \rangle=V.$
	Assume $s\neq 0.$ Notice that $b+s\geq \delta.$ Let $\tilde v_1,\dots,\tilde v_s$ be linearly independent elements of $U_1$ and $\tilde w_1,\dots,\tilde w_s$ linearly independent elements of $U_2.$ Moreover let $\tilde U_1=\langle 
	\tilde v_1,\dots,\tilde v_s\rangle$ and $\tilde U_2=\langle 
	\tilde w_1,\dots,\tilde w_s\rangle$. Since $|\tilde U_1 \cup \tilde U_2|\leq 2|F|^s-1<|F|^\delta,$ there exists $\tilde z \in V \setminus (\tilde U_1\cup \tilde U_2).$ Consider $\tilde x= (\tilde v_1,\dots,\tilde v_s, \tilde z)$ and
	$\tilde y= (\tilde w_1,\dots,\tilde w_s, \tilde z).$ Since $(s+1)+(b-1)\geq \delta$ and $\dim_F\langle  {\tilde v_1},\dots,\tilde v_s, \tilde z\rangle=\dim_F\langle  \tilde w_1,\dots,\tilde w_s, \tilde z\rangle=s+1,$ by induction there exist 
	$\tilde z_1,\dots,\tilde z_{b-1}$ such that $\langle \tilde v_1,\dots,\tilde v_s, \tilde z, \tilde z_1,\dots,\tilde z_{b-1}\rangle=\langle \tilde w_1,\dots,\tilde w_s, \tilde z, \tilde z_1,\dots,\tilde z_{b-1}\rangle=V.$ Clearly $z=(\tilde z, \tilde z_1,\dots,\tilde z_{b-1})$ satisfies the conditions  $\langle  v_1,\dots, v_a, \tilde z, \tilde z_1,\dots,\tilde z_{b-1}\rangle=\langle  w_1,\dots, w_a, \tilde z, \tilde z_1,\dots,\tilde z_{b-1}\rangle=V.$
\end{proof}


\begin{lemma}\label{prs}Let $F$ be a finite field and assume $\alpha\leq \beta.$ Given $R\in  M_{\alpha\times \beta}(F)$ and $S\in M_{\alpha\times \gamma}(F)$ consider the matrix 
	$\begin{pmatrix}R&S\end{pmatrix}\in M_{\alpha\times (\beta+\gamma)}$. Assume that
	$\ran \begin{pmatrix}R&S\end{pmatrix}=\alpha$ and let $\pi_{R,S}$ be the probability that a matrix $Z \in M_{\gamma\times \beta}(F)$ satisfies the condition $\ran(R+SZ)=\alpha.$ Then $$\pi_{R,S} > 1-\frac{q^\alpha}{q^\beta(q-1)}.$$
\end{lemma}

\begin{proof}
	There exist $m\leq \min\{\alpha,\gamma\},$  $X \in \GL(\alpha,F)$ and $Y \in \GL(\gamma,F)$ such that
	\[
	XSY=\begin{pmatrix}
	I_m&0_{m\times(\gamma-m)}\\
	0_{(\alpha-m)\times m}&0_{(\alpha-m)\times(\gamma-m)}
	\end{pmatrix},
	\] 
	where $I_m$ is the identity element in $M_{m\times m}(F).$
	Since
	\[
	\alpha=\ran \begin{pmatrix}R&S\end{pmatrix}=\ran\left( X\begin{pmatrix}R&S\end{pmatrix}\begin{pmatrix}I_\beta&0_{\beta\times\gamma}\\0_{\gamma\times\beta}&Y\end{pmatrix}\right)=\ran\begin{pmatrix}XR&XSY\end{pmatrix}
	\]
	and   
	\[
	\begin{aligned}
	\ran(R+SZ)&=\ran(X(R+SZ))=\ran(XR+XSZ)\\
	&= \ran(XR+XSY(Y^{-1}Z)),
	\end{aligned}
	\] 
	it is not restrictive (replacing $R$ by $XR$, $S$ by $XSY$ and $Z$ by $Y^{-1}Z$) to assume 
	\[
	S=\begin{pmatrix}
	I_m&0_{m\times (\gamma-m)}\\
	0_{(\alpha-m)\times m}&0_{(\alpha-m)\times (\gamma-m)}
	\end{pmatrix}.
	\]	
	Denote by $v_1,\dots,v_\alpha$ the rows of $R$ and by $z_1,\dots,z_\gamma$ the rows of $Z.$ The fact that the rows of $(R\ S)$ are linearly independent implies that $v_{m+1},\dots,v_\alpha$ are linearly independent vectors of $F^\beta.$ The condition $\ran(R+SZ)=\alpha$ is equivalent to asking that
	\[v_1+z_1,\dots,v_m+z_m,v_{m+1},\dots,v_\alpha\] are linearly independent.
	The probability that $z_1,\dots,z_m$ satisfy this condition is 
	\[
	\left(1-\frac{q^{\alpha-m}}{q^\beta}\right)\left(1-\frac{q^{\alpha-m+1}}{q^\beta}\right)\cdots \left(1-\frac{q^{\alpha-m+(m-1)}}{q^\beta}\right).
	\]
	Hence
	\[
	\begin{aligned}
	\pi_{R,S}&=\left(1-\frac{q^{\alpha-m}}{q^\beta}\right)\left(1-\frac{q^{\alpha-m+1}}{q^\beta}\right)\cdots \left(1-\frac{q^{\alpha-m+(m-1)}}{q^\beta}\right)\\&\geq 1-\frac{q^{\alpha-m}(1+q+\dots+q^{m-1})}{q^\beta}\\&=1-\frac{q^{\alpha-m}(q^m-1)}{q^\beta(q-1)}>1-\frac{q^\alpha}{q^\beta(q-1)}.\qedhere
	\end{aligned}
	\]   
\end{proof}

\begin{lemma}
	\label{que}
	Let $F$ be a finite field. Given positive integers $u,v,n,t$ satisfying  $n\leq \min\{u,v\}$ and $t+n=u+v$, suppose   that $A_1,A_2 \in M_{n\times u}(F)$,  $B\in  M_{n\times v}(F)$, $D_1, D_2 \in  M_{t\times u}(F)$ with the property that
	\[
	\ran \begin{pmatrix}B&A_1\end{pmatrix}=\ran \begin{pmatrix}B&A_2\end{pmatrix}=n,
	\]
	\[
	\ran \begin{pmatrix}A_1\\D_1\end{pmatrix}=\ran \begin{pmatrix}A_2\\D_2\end{pmatrix}=u.
	\]
	Then there exists $C\in M_{t\times v}(F)$ such that
	\[
	\det\begin{pmatrix}B&A_1\\C&D_1\end{pmatrix}\neq 0 \quad \text { and } \quad \det\begin{pmatrix}B&A_2\\C&D_2\end{pmatrix}\neq 0,
	\]
	except when $|F|=2$, $n=v$ and $\det B \neq 0$.
\end{lemma}

\begin{proof} Let $r=\ran(B).$ There exist $X \in \GL(n,F)$ and $Y\in \GL(v,F)$ such that
	\[
	XBY=\begin{pmatrix}
	I_r&0_{r\times (v-r)}\\
	0_{(n-r) \times r}&0_{(n-r) \times (v-r)}
	\end{pmatrix},
	\]
	where $I_r$ is the identity element in $M_{r\times r}(F).$
	Let $A_{11}, A_{21}\in M_{r\times u}(F)$ and $A_{12}, A_{22}\in M_{(n-r)\times u}(F)$
	such that
	\[
	XA_1=\begin{pmatrix}A_{11}\\A_{12}\end{pmatrix},\quad XA_2=\begin{pmatrix}A_{21}\\A_{22}\end{pmatrix}.
	\]
	For $i\in\{1,2\},$ since
	\[\begin{aligned}
	n=\ran \begin{pmatrix}B & A_i\end{pmatrix}&=
	\ran \left(X\begin{pmatrix}B&A_i\end{pmatrix}\begin{pmatrix}Y&0_{v \times u}\\0_{u\times v }&I_u\end{pmatrix}\right)\\ &=\ran \begin{pmatrix}
	I_r&0_{r\times (v-r)}&A_{i1}
	\\0_{(n-r)\times r}&0_{(n-r)\times (v-r)}&A_{i2}
	\end{pmatrix},
	\end{aligned}
	\]
	it must be $\ran (A_{i2})=n-r$.
	In particular there exists $Z_i\in \GL(u,F)$ such that
	\[
	XA_iZ_i=\begin{pmatrix}
	A_{i1}\\
	A_{i2}
	\end{pmatrix}Z_i=
	\begin{pmatrix}
	A_{i1}^*&A_{i2}^*\\
	0_{(n-r) \times u-(n-r)}&I_{n-r}
	\end{pmatrix},
	\]
	with $A_{i1}^* \in M_{r\times u-(n-r)}(F)$ and  $A_{i2}^* \in M_{r\times (n-r)}(F).$
	Notice that
	\[
	\begin{aligned}
	\det\begin{pmatrix}
	XBY&XA_iZ_i\\
	CY&D_iZ_i
	\end{pmatrix}&=
	\det\left(\begin{pmatrix}
	X&0_{n \times t}\\
	0_{t \times n}&I_t
	\end{pmatrix}
	\begin{pmatrix}
	B&A_i\\
	C&D_i
	\end{pmatrix}
	\begin{pmatrix}
	Y&0_{v \times u}\\
	0_{u \times v}&Z_i
	\end{pmatrix}
	\right)\\&=
	\det(X)\det(Y)\det(Z_i)\det
	\begin{pmatrix}
	B&A_i\\
	C&D_i
	\end{pmatrix}.
	\end{aligned}
	\]
	This means that it is not restrictive to assume
	\[
	B=\begin{pmatrix}
	I_r&0_{r \times (v-r)}\\
	0_{(n-r) \times r}&0_{(n-r) \times (v-r)}
	\end{pmatrix},\quad 
	A_i=\begin{pmatrix}
	A_{i1}^*&A_{i2}^*\\
	0_{(n-r) \times u-(n-r)}&I_{n-r}
	\end{pmatrix},
	\]
	with $A_{i1}^* \in M_{r\times  u-(n-r)}(F),$ $A_{i2}^* \in M_{r\times (n-r)}(F).$
	Let $C_1\in M_{t\times r}(F),  C_2\in M_{t\times v-r}(F),$ $D_{i1}\in M_{t\times u-(n-r)}(F), D_{i2}\in M_{t\times (n-r)}(F)$ such that 
	\[
	\begin{pmatrix}C_1&C_2\end{pmatrix}=C  \quad \text{and} \quad
	\begin{pmatrix}D_{i1}&D_{i2}\end{pmatrix}=D_i.
	\] Notice that
	\[
	\begin{aligned}& \det\begin{pmatrix}
	B&A_i\\
	C&D_{i}
	\end{pmatrix}
	=\det\begin{pmatrix}
	I_r&0_{r\times (v-r)}&A_{i1}^*&A_{i2}^*\\
	0_{(n-r)\times r}&0_{(n-r)\times (v-r)}&0_{(n-r) \times u-(n-r)}&I_{n-r}\\
	C_1&C_{2}&D_{i1}&D_{i2}
	\end{pmatrix}\\&=(-1)^{n-r}\det
	\begin{pmatrix}
	I_r&0_{r\times (v-r)}&A_{i1}^*\\
	C_1&C_{2}&D_{i1}
	\end{pmatrix}\\&=(-1)^{n-r}\det\left(
	\begin{pmatrix}
	I_r&0_{r\times(v-r)}&A_{i1}^*\\
	C_1&C_{2}&D_{i1}
	\end{pmatrix}
	\begin{pmatrix}
	I_r&0_{r\times (v-r)}&-A_{i1}^*\\
	0_{(v-r)\times r }&I_{v-r}&0_{(v-r) \times u-(n-r)}\\
	0_{u-(n-r)\times r}&0_{u-(n-r)\times (v-r)}&I_{u-(n-r)}
	\end{pmatrix}\right)\\&=(-1)^{n-r}\det
	\begin{pmatrix}
	I_r&0_{r\times (v-r)}&0_{r\times u-(n-r)}\\
	C_1&C_{2}&D_{i1}-C_1A_{i1}^*
	\end{pmatrix}\\&=(-1)^{n-r}\det
	\begin{pmatrix}
	C_{2}&D_{i1}-C_1A_{i1}^*
	\end{pmatrix}.
	\end{aligned}
	\]
	Assume that we can find $C_1$ such that 
	\[
	\ran (D_{11}-C_1A_{11}^*) = \ran (D_{21}-C_1A_{21}^*)= u- (n-r)
	\]
and let $W_1, W_2$ be the subspaces of $F^t$ spanned, respectively, by the columns of the two matrices $D_{11}-C_1A_{11}^*$ and $D_{21}-C_1A_{21}^*$. By Lemma \ref{comcom}, there exists a subspace $U$ of $F^t$ such that $F^t=W_1\oplus U=W_2\oplus U.$ If $C_2$ is a matrix whose columns are a basis for $U,$ then  
	$$\det\begin{pmatrix}C_{2}&D_{11}-C_1A_{11}^*\end{pmatrix}\neq 0
	\text { and } \det\begin{pmatrix}C_{2}&D_{21}-C_1A_{21}^* \end{pmatrix}\neq 0
	$$ and $C=(C_1 \ C_2)$ is a matrix with the desired property.
	Set 
	\[
	R_1=D_{11}^\text{T},\ R_2=D_{21}^\text{T},\ S_1=A_{11}^{*\text{T}},\ S_2=A_{21}^{*\text{T}},\ Z= -C_1^\text{T}.
	\]
	The previous observation implies that a matrix $C$ with the requested properties exists if, and only if, there exists $Z\in M_{r\times t}(F)$ such that
	\begin{equation}
	\ran(R_1+S_1Z)=\ran(R_2+S_2Z)=u-(n-r).
	\end{equation} 
	Notice that $R_1, R_2\in M_{u-(n-r)\times t}(F)$, $S_1, S_2\in M_{u-(n-r)\times r}(F)$ have the property that 
	\[
	\ran \begin{pmatrix}R_1&S_1\end{pmatrix}=\ran \begin{pmatrix}R_2&S_2\end{pmatrix}=u-(n-r).
	\]
	If either $|F|=q>2$ or $u-(n-r)<t$, then, by applying Lemma \ref{prs} with $\alpha=u-(n-r), \beta=t,\gamma=r$, we have
	\[
	\pi_{R_1,S_1}>\frac{1}{2} \quad \text { and }\quad  \pi_{R_2,S_2}>\frac{1}{2}
	\]
	and this is sufficient to ensure that a matrix $Z$ with the requested property exists. Therefore we may assume $u-(n-r)=t$  and $q=2.$ This implies that $v=r$, and so that $v=n=r$, i.e. $\det B \neq 0$. This concludes the proof. 
\end{proof}

The main ingredient in the proof of our results about the diameter of $\Gamma^*_{a,b}(G)$ is the theory of crowns, introduced by Gasch\"utz in \cite{Ga2}.
We recall some properties of the crowns of a finite soluble group. 
Let $G$ be a finite soluble group, and let $\mathcal V_G$ be a set
of representatives for the irreducible $G$-groups that are
$G$-isomorphic to a complemented chief factor of $G$. For $V \in
\mathcal V_G$ let $R_G(V)$ be the smallest normal subgroup contained
in $C_G(V)$ with the property that $C_G(V)/R_G(V)$ is
$G$-isomorphic to a direct product of copies of $V$ and it has a
complement in $G/R_G(V)$. The factor group $C_G(V)/R_G(V)$ is
called the $V$-crown of $G$. The non-negative integer
$\delta_G(V)$ defined by $C_G(V)/R_G(V)\cong_G V^{\delta_G(V)}$ is
called the $V$-rank of $G$ and it coincides with the number of
complemented factors in any chief series of $G$ that are
$G$-isomorphic to $V$. If $\delta_G(V) \neq 0$, then the $V$-crown
is the socle of $G/R_G(V)$.

\begin{prop}\label{prouno}\cite[Proposition 2.4]{LM2}
	\label{lemma} Let $G$ and $\mathcal V_G$ be as above. Let $x_{1},
	\ldots , x_{u}$ be elements of $G$ such that $\langle
	x_1,\dots,x_u,R_G(V)\rangle=G$ for any $V \in \mathcal V_G$. Then
	$\langle x_1,\dots,x_u\rangle=G$.
\end{prop}

\begin{lemma}{\cite[Lemma 1.3.6]{classes}}\label{corona}
	Let $G$ be a finite soluble group with trivial Frattini subgroup. There exists
	a crown $C/R$ and a non-trivial normal subgroup $U$ of $G$ such that $C=R\times U.$
\end{lemma}

\begin{lemma}{\cite[Proposition 11]{crowns}}
	\label{sotto} 
	Assume that $G$ is a finite soluble group with trivial Frattini subgroup and let $C, R, U$ be as in the statement of Lemma \ref{corona}. If $HU=HR=G,$ then $H=G.$
\end{lemma}

Now let $V$ be a finite dimensional vector space over a finite field
of prime order. Let $K$ be a $d$-generated linear soluble group acting
irreducibly and faithfully on $V$ and fix a generating $d$-tuple $(k_1,\dots,k_d)$ of $K.$
For a positive integer $u$ we consider
the semidirect product $G_u = V^u \rtimes K$, where $K$ acts in the
same way on each of the $u$ direct factors. We will use the aforementioned properties  of the crowns, in particular Proposition \ref{prouno} and Lemmas \ref{corona} and \ref{sotto}, to essentially reduce the study of the graph $\Gamma^*_{a,b}(G)$ to the particular case when $G\cong G_u.$
Put $F =
\mathrm{End}_{K}(V)$. Let $n$ be the dimension of $V$ over $F$. We may identify $K =
\langle k_1, \dots, k_d \rangle$ with a subgroup of the general linear group $\GL(n,F)$. In this
identification $k_i$ becomes an $n\times n$ matrix $X_i$ with coefficients in $F$; denote by $A_i$ the matrix
$I_n-X_i.$ Let $w_i=(v_{i,1},\dots,v_{i,u})\in V^u.$
Then every $v_{i,j}$ can be
viewed as a $1 \times n$ matrix. Denote the $u \times n$ matrix
with rows $v_{i,1},\dots,v_{i,u}$  by $D_i$. The following result is proved in
\cite[Section 2]{CL4}.

\begin{prop}\label{richiami} The group $G_u=V^u\rtimes K$ can be generated by $d$ elements if and only if $u\leq n(d-1).$ Moreover
	\begin{enumerate}
		\item $\mathrm{rank} \begin{pmatrix}A_1&\dots&A_d\end{pmatrix}=n.$ \item
		$\langle k_1w_1,\dots,k_dw_d \rangle=V^u \rtimes K$ if and only if
		$\mathrm{rank} \begin{pmatrix}A_1&\cdots&A_d\\
		D_1&\cdots&D_d\end{pmatrix} = n+u.$
	\end{enumerate}
\end{prop}

The next result may seem rather technical, but it provides crucial information on the graph $\Gamma_{a,b}(G)$ when $G\cong V^\delta\rtimes K.$

\begin{prop}\label{corfour}
	%
	
	Let $K$ be a non-trivial d-generator linear soluble group acting irreducibly and faithfully on $V$ and consider the semidirect product $G=V^\delta \rtimes K$ with $\delta\leq n(d-1)$, where $n=\dim_{\End_G(V)}V.$ Let $a$ and $b$ be non-negative integers such that $a+b=d$,  $s\in \{a,b\}$ and $t=d-s$.  Assume that   $(t, |F|)\neq (1,2)$ and there exist, for $i\in\{1,2\}$, $x_{i1},\ldots,x_{is}$ and $y_1,\ldots,y_t$ in $K$, and $w_{i1},\ldots,w_{is}$ in $V^\delta$ such that
	\begin{enumerate}
		\item $(x_{11}w_{11},\ldots, x_{1s}w_{1s})$ and $(x_{21}w_{21},\ldots, x_{2s}w_{2s})$ are non-isolated vertices belonging to $V_s$ in  the graph  $\Gamma^*_{a,b}(G)$,
		\item $\langle x_{11},\ldots, x_{1s}, y_1,\ldots,y_t\rangle=\langle y_1,\ldots,y_t, x_{21},\ldots, x_{2s} \rangle=K.$
	\end{enumerate}
	Then there exist $w_1, \ldots,w_t \in  V^\delta$ with
	$$\langle x_{11},\ldots, x_{1s},y_1w_1,\ldots,y_tw_t  \rangle=\langle y_1w_1,\ldots,y_tw_t, x_{21},\ldots, x_{2s}\rangle=G.$$
\end{prop}

\begin{proof}
	Since $V^\delta\rtimes K$ is an epimorphic image of $V^{n(d-1)}\rtimes K$, it suffices to prove the statement in the particular case where $G=V^{n(d-1)}\rtimes K.$
	We may identify the elements $x_{i1},\ldots,x_{is},y_1,\ldots,y_t$ with matrices  $ X_{i1},\ldots,X_{is},Y_1,\ldots,Y_t\in \GL(n,F)$, respectively, where $F=\End_G(V)$ and  $w_{i1},\ldots,w_{is},w_1, \ldots,w_t\in V^{n(d-1)}$ with  matrices $D_{i1},\ldots,D_{is}$ and $C_1, \ldots,C_t$ in
	$M_{n(d-1)\times n}(F)$, respectively. We now apply Proposition \ref{richiami}. Let
$$\begin{aligned}A_{ij}&=I_n-X_{ij}, \quad \text{for}\, i\in\{1,2\}\, \text{and}\, j\in\{1,\ldots, s\},
\\B_k&=I_n-Y_k, \quad \text{for}\, k\in\{1,\ldots, t\}.\end{aligned}$$
	Conditions (1) and (2)  imply that 	$$\ran (A_{11} \ \ldots \ A_{1s}\ B_1 \ \ldots \ B_t )=\ran (A_{21} \ \ldots  \ A_{2s}\ B_1 \ \ldots  \ B_t)=n$$ and
	$$\ran \begin{pmatrix}A_{11} \ \ldots \ A_{1s}\\ D_{11}\ \ldots \ D_{1s} \end{pmatrix}=\ran \begin{pmatrix}A_{21} \ \ldots  \ A_{2s} \\ D_{21}\ \ldots \ D_{2s}\end{pmatrix}=ns.$$ Moreover our statement is equivalent to saying that there exist  $t$ matrices $C_1,\ldots,C_t \in M_{n(d-1)\times n}(F)$
	with
	$$\det\begin{pmatrix}A_{11} \ \ldots \ A_{1s}& B_1\ \ldots \ B_t\\ D_{11} \ \ldots \ D_{1s}& C_1\ \ldots \ C_t\end{pmatrix}\neq 0,\quad \det\begin{pmatrix}B_1\ \ldots \ B_t& A_{21} \ \ldots \ A_{2s} \\ C_1\ \ldots \ C_t& D_{21} \ \ldots \ D_{2s}\end{pmatrix}\neq 0. 	
	$$
	Put, for $i\in \{1,2\}$ 
$$\begin{aligned}
	A_i&= (A_{i1} \ \ldots \ A_{is})\in M_{n\times ns}(F),\\D_i&= (D_{i1} \ \ldots \ D_{is})\in M_{n(d-1)\times ns}(F),\\
	B&=(B_1 \ \ldots \ B_t)\in M_{n\times nt}(F).
\end{aligned}$$
	The existence of $C=(C_1 \ \ldots \ C_t)\in M_{n(d-1)\times nt}(F)$ such that 
	$$ \det\begin{pmatrix}A_{1}& B\\ D_1& C\end{pmatrix}\neq 0,\ \det\begin{pmatrix}B& A_2 \\ C& D_2\end{pmatrix}\neq 0$$
	is ensured by Lemma \ref{que}.
	Notice that the fact that $K$ is a non-trivial subgroup of $\GL(n,F)$ implies that
	$n\geq 2$ if $|F|=2$. Moreover if $|F|=2$ and  $\ran B=\ran(B_1 \ \ldots \ B_t)=nt$, we necessarily have $t=d-s=1$.
\end{proof}

We are now ready to prove the main result of this section.

\begin{thm}\label{abd}
	Let $G$ be a finite soluble group,  $a$ and $b$ be non-negative integers, $s\in \{a,b\}$ and $t=a+b-s$. Assume that either $t\neq 1$ or
	$G$ has the following property: if $A$ is
	a non-trivial irreducible $G$-module  $G$-isomorphic to a complemented chief factor of $G$, then $|\End_G(A)|>2$ (this holds in particular when the derived subgroup of $G$ is either nilpotent or of odd order).
	Then  in the graph $ \Gamma^*_{a,b}(G)$ given  any two vertices $x_1,x_2\in V_s$,  there exists $y\in V_{t}$ which is adjacent to both $x_1$ and $x_2$. 
\end{thm}

\begin{proof}
	We may assume $d:=a+b\geq d(G)$. We argue by induction on the order of $G$. Choose two  vertices 
	$x_1=(x_{11},\ldots,x_{1s})$ and $x_2=(x_{21},\ldots,x_{2s})$ in $V_s$. Let $F=\frat(G)$ be the Frattini subgroup of $G$. Clearly $x_1F=(x_{11}F,\ldots,x_{1s}F)$ and $x_2F=(x_{21}F,\ldots,x_{2s}F)$ are vertices of the graph  $ \Gamma^*_{a,b}(G/F).$ If
	$F\neq 1,$ then, by induction, there exists a  $t$-tuple $yF=(y_1F,\ldots,y_tF)$ which is simultaneously adjacent to  $x_1F$ and $x_2F$  in the graph $ \Gamma^*_{a,b}(G/F)$.
		This implies that  $G=\langle x_{11},\ldots,x_{1s} , y_1,\ldots,y_t \rangle F= \langle x_{21},\ldots,x_{2s} , y_1,\ldots,y_t \rangle F
	=\langle x_{11},\ldots,x_{1s} , y_1,\ldots,y_t\rangle=\langle x_{21},\ldots,x_{2s} , y_1,\ldots,y_t\rangle$, hence $y=(y_1,\ldots,y_t)$ is a $t$-tuple adjacent to both $x_1$ and $x_2$ in $ \Gamma^*_{a,b}(G)$.  Therefore we may assume $F=1.$ In this case, by Lemma \ref{corona}, there exist a crown $C/R$ of $G$ and
	a normal subgroup $U$ of $G$ such that $C=R\times U.$
	We have $R=R_G(A)$ where $A$ is an irreducible $G$-module  and $U\cong_G A^\delta$ for $\delta=\delta_G(A).$ 
	By induction,  in the graph $ \Gamma^*_{a,b}(G/U)$, there exists a $t$-tuple $yU=(y_1U,\ldots,y_tU)$ which is adjacent to  both   $x_1U=(x_{11}U,\ldots,x_{1s}U)$ and $x_2U=(x_{21}U,\ldots,x_{2s}U)$. In particular  we have
	\begin{equation}\label{modu}
	\langle x_{11},\ldots,x_{1s} , y_1,\ldots,y_t \rangle U=\langle x_{21},\ldots,x_{2s} , y_1,\ldots,y_t \rangle U=G.
	\end{equation}
	
	We work in the factor group $\bar G=G/R.$ We have $\bar C=C/R=UR/R\cong U\cong A^\delta$
	and either $A\cong C_p$ is a trivial $G$-module and $\bar G\cong (C_p)^\delta$  or $\bar G= \bar U \rtimes \bar H \cong A^\delta \rtimes K$
	where $K \cong \bar H$ acts in the same way on each of the $\delta$ factors of $A^\delta$ and this action
	is faithful and irreducible. Since $\bar G$ is $d$-generated, we have $\delta\leq d$ if $A$ is a trivial $G$-module,  $\delta\leq n (d-1)$, where $n=\dim_{{\End_G(A)}}A$
	otherwise.

	By Lemma \ref{lemmanil} in the first case  and by Proposition \ref{corfour} in the second case, there exist $u_1,\ldots,u_t\in U$ with
	$$\langle \bar x_{11},\ldots, \bar x_{1s}, \bar y_1\bar u_1, \ldots ,\bar y_t\bar u_t\rangle= \langle \bar y_1\bar u_1, \ldots, \bar y_t\bar u_t, \bar x_{21},\ldots, \bar x_{2s} \rangle=\bar G,$$ i.e.
	\begin{equation}\label{modr}
	\langle  x_{11},\ldots,  x_{1s},  y_1u_1, \ldots,  y_tu_t\rangle R= \langle  y_1u_1, \ldots,  y_tu_t,  x_{21},\ldots,  x_{2s} \rangle R=G.
	\end{equation}
	In view of Lemma \ref{sotto}, from (\ref{modu}) and (\ref{modr}), we obtain that  $$\langle  x_{11},\ldots,  x_{1s},  y_1u_1, \ldots,  y_tu_t\rangle= \langle  y_1u_1, \ldots,  y_tu_t,  x_{21},\ldots,  x_{2s} \rangle =G. \qedhere$$ 
	\end{proof}
Now from Theorem \ref{abd} and \cite[Theorem 1]{diam} we  easily deduce the following result. 
\begin{cor}
	\label{diametro}
Let  $G$ be a finite soluble group and let $a$ and $b$ be non-negative integers. Then $$\diam(\Gamma^*_{a,b}(G))\leq 4.$$ Moreover
\begin{enumerate}
\item Assume $a=b$. If either $G$  has the property that $|\End_G(V)|>2$ for every non-trivial irreducible $G$-module $V$ which is $G$-isomorphic to a complemented chief factor of $G$ or $a\neq 1$, then $\diam(\Gamma^*_{a,a}(G))\leq 2$. Otherwise $\diam(\Gamma^*_{a,a}(G))\leq 3$.
\item Assume $a<b.$ If either $G$  has the property that $|\End_G(V)|>2$ for every non-trivial irreducible $G$-module $V$ which is $G$-isomorphic to a complemented chief factor of $G$ or $a\neq 1$, then $\diam(\Gamma^*_{a,b}(G))\leq 3$.

\end{enumerate}
\end{cor}

In the remaining part of this section we want to prove that Theorem \ref{abd} does not remain true, when $t=1$, if we drop out the assumption that $G$ has the  property that  $|\End_G(A)|>2$ whenever $A$ is
a non-trivial irreducible $G$-module $G$-isomorphic to a complemented chief factor of $G$. Indeed we want show that, for every $d\geq 2,$ it can be constructed a $d$-generator soluble group $G$ with the property that $\Gamma^*_{1,d-1}(G)$ 
contains two distinct vertices $\alpha_1=(g_{1,1},\dots,g_{1,d-1})$ and $\alpha_2=(g_{2,1},\dots,g_{2,d-1})$ without a common adjacent vertex.
First we note that Proposition \ref{richiami} has the following corollary.
\begin{cor}\label{coro}
	Let $d$ be a positive integer with $d\geq 2$, let $V=\mathbb F_2 \times \mathbb F_2,$ where $\mathbb F_2$ is the field with 2 elements and let $\Gamma=\GL(2,2)\ltimes V^{u}$ with $u=2(d-1).$ Assume that $\langle k_1,\dots,k_d \rangle =\GL(2,2)$ and let $\gamma_1=k_1(v_{11},\dots,v_{1u}),\dots, \gamma_d=k_d(v_{d1},\dots,v_{du})$ in $\Gamma.$ We have $\Gamma=\langle \gamma_1, \dots,\gamma_d\rangle$ if and only if $$\begin{pmatrix}1-k_1&\dots&1-k_d\\v_{11}&\dots&v_{d1}\\\dots&\dots&\dots\\v_{1u}&\dots&v_{du}\end{pmatrix}\neq 0.$$
\end{cor}

Now let $H=\GL(2,2)\times \GL(2,2)$ and let $$W=(V_{11}\times \dots \times V_{1u})\times (V_{21}\times \dots \times V_{2u})$$ be the direct product of $2u$ 2-dimensional vector spaces over the field $\mathbb F_2$ with two elements. We define an action of $H$ on $W$ by setting
$$((v_{11},\dots,v_{1u}),(v_{21},\dots,v_{2u}))^{(x,y)}=((v_{11}^x,\dots,v_{1u}^x),(v_{21}^y,\dots,v_{2u}^y))$$
and we consider the semidirect product $G=H\ltimes W.$
Let $$\begin{aligned}
N_1:=&C_G(V_{21})=\dots=C_G(V_{2u})=\{(k,1)\mid k\in \GL(2,2)\},\\
N_2:=&C_G(V_{11})=\dots=C_G(V_{1u})=\{(1,k)\mid k\in \GL(2,2)\}.\\
\end{aligned}$$
A set of representatives for the $G$-isomorphism classes of
the complemented chief factors of $G$ contains precisely 5 elements:
\begin{itemize}
	\item $Z,$ a central $G$-module of order 2, with $R_G(Z)=G^\prime=\ssl(2,2)^2\ltimes W$;
	\item $U_1,$ a non-central $G$-module of order
	3, with $R_G(U_1)=N_2\ltimes W$;
	\item $U_2,$ a non-central $G$-module of order
	3, with $R_G(U_2)=N_1\ltimes W$;
	\item $V_{11}$, with $R_G(V_{11})=V_{21}\times \dots \times V_{2u} \times N_2$;
	\item $V_{21}$, with $R_G(V_{21})=V_{11}\times \dots \times V_{1u} \times N_1$.
\end{itemize}
Let $$
\begin{aligned}(x_1,y_1)((v_{111},\dots,v_{11u}),(v_{121},\dots,v_{12u}))&=g_1,\\
\dots\dots \dots\dots \dots\dots \dots\dots \dots\dots \dots \dots \dots &  \dots \dots\\ (x_d,y_d)((v_{d11},\dots,v_{d1u}),(v_{d21},\dots,v_{d2u}))&=g_d.\end{aligned}$$
We want to apply  Proposition \ref{prouno} to check whether $\langle g_1,\dots,g_d \rangle=G$. The three conditions
$$\langle g_1,\dots,g_d \rangle R_G(Z)=G, \langle g_1,\dots,g_d\rangle R_G(U_1)=G,  \langle g_1,\dots,g_d\rangle R_G(U_2)=G$$ are equivalent to
$\langle g_1,\dots,g_d\rangle W=G$, i.e. to $\langle (x_1,y_1),\dots, (x_d,y_d)\rangle=H.$ Moreover 
$\langle g_1,\dots,g_d\rangle R_G(V_{11})=G$ { if and only if } 
$$\langle x_1(v_{111},\dots,v_{11u}),\dots, x_d(v_{d11},\dots,v_{d1u})\rangle=(V_{11}\times \dots \times V_{1u})\rtimes \GL(2,2),$$ 
$\langle g_1,\dots,g_d\rangle R_G(V_{21})=G$ { if and only if } 
$$\langle y_1(v_{121},\dots,v_{12u}),\dots, y_d(v_{d21},\dots,v_{d2u})\rangle=(V_{21}\times \dots \times V_{2u})\rtimes \GL(2,2).$$ 
Applying Corollary \ref{coro} we conclude that
$$\langle g_1,\dots,g_d \rangle=G$$ if and only if the following conditions are satisfied:
$$\begin{aligned}(1)&\quad \quad \quad \quad \quad
\langle (x_1,y_1), \dots,(x_d,y_d)\rangle=H=\GL(2,2)\times \GL(2,2),\\
(2)&\quad \quad \quad \quad \quad\det\begin{pmatrix}1-x_1&\dots&1-x_d\\v_{111}&\dots&v_{d11}\\\dots&\dots&\dots\\v_{11u}&\dots&v_{d1u}
\end{pmatrix}\neq 0,\\
(3)&\quad \quad \quad \quad \quad\det\begin{pmatrix}1-y_1&\dots&1-y_d\\v_{121}&\dots&v_{d21}\\\dots&\dots&\dots\\v_{12u}&\dots&v_{d2u}
\end{pmatrix}\neq 0.
\end{aligned}$$
Consider the following elements of $\GL(2,2)$:
$$x:=\begin{pmatrix}1&0\\1&1\end{pmatrix},\quad
y:=\begin{pmatrix}1&1\\1&0\end{pmatrix}, \quad z:=\begin{pmatrix}1&1\\0&1\end{pmatrix},$$
and the following elements of $\mathbb F_2^2$:
$$0=(0,0),\quad e_1=(1,0),\quad e_2=(0,1).$$
Let
$$\begin{aligned}a_{11}:=&(x,x)((0,e_2,0,\dots,0)),(0,e_2,0,\dots,0)),\\ a_{12}:=&(x,x)((e_1,e_2,0,\dots,0),(e_1,e_2,0,\dots,0)),\\
a_2:=&((0,0,e_1,e_2,0,\dots,0),(0,0,e_1,e_2,0,\dots,0)),\\
&\dots \dots \dots \dots\dots\dots \dots \dots \dots\dots \dots \dots \dots\dots\\
a_{d-1}:=&((0,\dots,0,e_1,e_2),(0,\dots,0,e_1,e_2)),\\
b_1:=&(y,z)((e_1,0,\dots,0),(e_1,0,\dots,0)),\\ b_2:=&(y,z)((0,\dots,0),(e_1,0,\dots,0)).
\end{aligned}$$
It can be easily checked that
either $a_{11},a_2,\dots,a_{d-1},b_1$ as  $a_{12},a_2,\dots,a_{d-1},b_2$ satisfy the three conditions (1), (2) (3) and therefore $$\langle a_{11},a_2,\dots,a_{d-1},b_1\rangle=\langle a_{12},a_2,\dots,a_{d-1},b_2\rangle=G.$$
Now we want to prove that there is no $b\in G$ with  
$$\langle a_{11},a_2,\dots,a_{d-1},b\rangle=\langle a_{12},a_2,\dots,a_{d-1},b\rangle=G.$$

Let $b=(h_1,h_2)((v_{11},\dots,v_{1u}),(v_{21},\dots,v_{2u}))
$, and assume by contradiction that $\langle a_{11},a_2,\dots,a_{d-1},b\rangle=\langle a_{12},a_2,\dots,a_{d-1},b\rangle=G.$ We must have in particular that condition (1) holds, i.e. $\langle (x,x), (h_1,h_2)\rangle=H.$ Since $(x,x)$ has order 2 and $H$ cannot be generated by two involutions (otherwise it would be a dihedral group) at least one of the two elements $h_1, h_2$ must have order 3: it is not restrictive to assume $h_1=y.$
Let $$A=\begin{pmatrix}1-x&0_{2\times 2}&\cdots&0_{2\times 2}\end{pmatrix}=\begin{pmatrix}0&0&0&\cdots&0\\1&0&0&\dots&0
\end{pmatrix},\quad B=1-y=\begin{pmatrix}0&1\\1&1
\end{pmatrix}$$
$$C_1=\begin{pmatrix}\begin{matrix}0&0\\0&1\end{matrix}&0_{2\times u-2}\\0_{u-2\times 2}&I_{u-2}\end{pmatrix},\quad  C_2=\begin{pmatrix}\begin{matrix}1&0\\0&1\end{matrix}&0_{2\times u-2}\\0_{u-2\times 2}&I_{u-2}\end{pmatrix},$$ $$D=\begin{pmatrix}v_{11}\\\vdots\\v_{1u}\end{pmatrix}=\begin{pmatrix}D_1\\D_2\end{pmatrix} \text{ with }D_2\in M_{u-2\times 2}(\mathbb F_2) \text \and D_1=\begin{pmatrix}
\alpha&\beta\\\gamma&\delta
\end{pmatrix}.$$
Conditions (2) 
must be satisfied, hence we must have
$$\det \begin{pmatrix}A&B\\C_1&D\end{pmatrix}=\det \begin{pmatrix}A&B\\C_2&D\end{pmatrix}=1.$$
However 
$$\det \begin{pmatrix}A&B\\C_1&D\end{pmatrix}=\det\begin{pmatrix}
\begin{matrix}0&0\\1&0
\end{matrix}&0_{2\times u-2}&\begin{matrix}0&1\\1&1
\end{matrix}\\\begin{matrix}0&0\\0&1\end{matrix}&0_{2\times u-2}&D_1\\0_{u-2\times 2}&I_{u-2}&D_2
\end{pmatrix}=
\det\begin{pmatrix}0&0&0&1\\1&0&1&1\\0&0&\alpha&\beta\\0&1&\gamma&\delta\end{pmatrix}=\alpha,$$
$$\det \begin{pmatrix}A&B\\C_2&D\end{pmatrix}=\det\begin{pmatrix}
\begin{matrix}0&0\\1&0
\end{matrix}&0_{2\times u-2}&\begin{matrix}0&1\\1&1
\end{matrix}\\\begin{matrix}1&0\\0&1\end{matrix}&0_{2\times u-2}&D_1\\0_{u-2\times 2}&I_{u-2}&D_2
\end{pmatrix}=
\det\begin{pmatrix}0&0&0&1\\1&0&1&1\\1&0&\alpha&\beta\\0&1&\gamma&\delta\end{pmatrix}=\alpha+1.$$
However, since $\alpha\in \mathbb F_2$ either $\alpha=0$ or $\alpha+1=0,$
so there is no $b\in G$ with  $\langle a_{11},a_2,\dots,a_{d-1},b\rangle=\langle a_{12},a_2,\dots,a_{d-1},b\rangle=G.$

\

We conclude this section noticing  that Theorem \ref{abd} can be applied to bound the diameter of the swap graph.

\begin{thm}\label{swapdiam} Suppose that a finite soluble group $G$ has the following property: if $A$ is
	a non-trivial irreducible $G$-module $G$-isomorphic to a complemented chief factor of $G$, then $|\End_G(A)|>2$ (this holds in particular when the derived subgroup of $G$ is either nilpotent or of odd order).
	If $d\geq d(G),$ then the diameter of the swap graph $\Sigma_d(G)$ is at most $2d-1.$ 
\end{thm} 
\begin{proof}
Assume  that $G=\langle a_1, \ldots, a_d\rangle=\langle b_1,\ldots, b_d\rangle.$ By Theorem \ref{abd}, there exists $x_1\in G$ such that $G=\langle x_1, a_2, \ldots,a_d\rangle=\langle x_1,b_2,\ldots,b_d\rangle.$  Applying $d-1$ times Theorem \ref{abd}, we find elements $x_i$, for $1\leq i \leq d-1$ satisfying 
$$G=\langle x_1, \ldots,x_{i-1}, x_i , a_{i+1},\ldots,a_d\rangle=\langle x_1, \ldots,x_{i-1}, x_i , b_{i+1},\ldots,b_d\rangle.$$  Hence  $\Sigma_d(G)$ contains the following path of length $2d-1$:
\begin{equation*}
\begin{matrix}
(a_1, \ldots, a_d), \\
(x_1, a_2,\ldots, a_d),\\
(x_1,x_2,a_3,\ldots,a_d),\\
\vdots \\
(x_1,\ldots,x_{d-1},a_d),\\
(x_1,\ldots,x_{d-1},b_d),\\
(x_1,\ldots,x_{d-1},b_{d-1},b_d),\\
\vdots \\
(x_1,b_2,\ldots,b_d),\\
(b_1,\ldots,b_d).\\
\end{matrix}
\end{equation*}
Since this path has length $2d-1$, we are done.
\end{proof}

\section{Direct powers of simple groups}\label{powsim}

In this section we will try to generalize some results proved in \cite{pwesim} concerning the generating graph of direct powers of non-abelian simple groups. As a by-product, we will see that the bounds on the diameter of the graphs $\Gamma^*_{a,b}(G)$, proved in Section \ref{boundsol}, does not remain true if we drop the solubility assumption: for every positive integer $\eta$ and every pair $a, b$ of positive integers, a finite group $G$ can be constructed such that $d(G)=a+b$ and $\Gamma^*_{a,b}(G)$ is connected with diameter at least $\eta.$

Let $S$ be a non-abelian finite simple group and denote by $A$ the automorphism group $\aut(S)$
of $S$. As usual we identify $S$ with the subgroup of $A$ consisting of the inner automorphisms. Let $d\geq 2$ be a positive integer and define
$\tau=\tau_d(S)$ to be the largest positive integer $r$ such
that $S^{r}$, the direct product of $r$ copies of $S,$ can be
generated by $d$ elements. Notice that the group $S^{r}$ cannot be
generated by $d$ elements whenever $r$ is larger than the number
of $A$-orbits on the set of $d$-tuples generating
$S$.
Actually,
$\tau$ is equal to the number of $A$-orbits on
ordered $d$-tuples of generators for $S$ and, for arbitrary elements
$x_1=(x_{1,1},\dots,x_{1,\tau}),\dots,x_d=(x_{d,1},\dots,x_{d,\tau})$ of $S^\tau,$ we
have that $S^\tau = \langle x_1,\dots,x_d \rangle$ if and only if the $d$-tuples
$(x_{1,i},\dots,x_{d,i})$ are distinct representatives for these orbits for $1 \leq i \leq \tau$.
Let $K=\aut(S^\tau).$ Recall that $K\cong A \wr
\mathrm{Sym}(\tau)$. Clearly $K\leq \aut(\Gamma_{a,b}(S^\tau))$ for every $a,b$ with $a+b=d$. The following easy remark will play a crucial role in our discussion.

\begin{lemma}\label{pi} Assume $S^\tau = \langle x_1,\dots,x_d \rangle.$ If $S^\tau=\langle
	y_1, \dots, y_d \rangle,$ then there exists $k \in K$ such that $(y_1, \dots, y_d)=(x_1^{k},\dots,x_d^{k})$.
\end{lemma}

\begin{proof} Assume $x_i=(x_{i,1},\dots,x_{i,\tau}),$ $y_j=(y_{j,1},\dots,y_{j,\tau})$ for $1\leq i,j \leq d.$
	Both $(x_{1,1},\dots,x_{d,1}),\dots,$ $(x_{1,\tau},\dots,x_{d,\tau})$ and
	$(y_{1,1},\dots,y_{d,1}),\dots,$ $(y_{1,\tau},\dots,y_{d,\tau})$	
	form a set of
	representatives for the $A$-orbits of the set of
	generating $d$-tuples for $S$. So there exist $\pi \in \perm(\tau)$ and $(a_1,\dots,a_\tau)\in A^\tau$ such that
	$(y_{1,i\pi},\dots,y_{d,i\pi})=(x_{1,i},\dots,x_{d,i})^{a_i}$ for each $i \in \{1,\dots,\tau\}.$ It follows that $(y_1, \dots, y_d)=(x_1^{k},\dots,x_d^{k})$ for $k=(a_1,\dots,a_\tau)\pi \in K.$
\end{proof}

\begin{cor}\label{arctrans} Let $\tau=\tau_{a+b}(S).$ Then the graph $\Gamma^*_{a,b}(S^\tau)$ is
	edge-transitive.
\end{cor}

Now we will introduce other notations, useful to study  the graph $\Gamma_{a,b}(S^\tau).$ 
Fix a vertex $x=(x_1,\dots,x_a)$ in the part $V_a$ of $\Gamma^*_{a,b}(S^\tau)$ corresponding to  the $(a)$-tuples and a vertex  $y=(y_1,\dots,y_b)$ in the part $V_b$ corresponding to  the $(b)$-tuples and let $C=C_K(x)$ and $D=C_K(y).$ 
To describe more precisely $C$  we need the following information.
Let $s_1,\dots,s_u$ be a set of representatives for the $A$-orbits of
$S^a$ that can be completed to a generating $d$-tuple of $S$. Every vertex $x\in V_a$ can be viewed as an $a\times \tau$ matrix $(x_{i,j})$ with $x_{i,j}\in S.$ Denote by $\tau_i$ the number of columns of $x$ that are $A$-conjugate to $s_i$. By Corollary \ref{arctrans} this number is independent on the choice of $x.$ In particular $x$ is $K$-conjugate to
$\bar x$ with
$$\bar x=(\underbrace{s_1,\ldots,s_1}_{\tau_1 \mbox{ terms}},\underbrace{s_2,\ldots,s_2}_{\tau_2 \mbox{ terms}},\dots,\underbrace{s_u,\ldots,s_u}_{\tau_u \mbox{ terms}}).$$
It follows that $C\cong C_K(\bar x)= \prod_{1\leq i \leq u}C_{A}(s_i)\wr \mathrm{Sym}(\tau_i)$. Clearly we have a similar description for $D=C_K(y)$, with the only difference that the role of $s_1,\dots,s_u$ will be played by a set of representatives $t_1,\dots,
t_v$  for the $A$-orbits of
$S^b$ that can be completed to a generating $d$-tuple of $S$.

\begin{lemma}\label{neig}Assume $1\leq a \leq b$ with $(a,b)\neq (1,1).$
	For every $i\neq 1$, there exists $\bar y_i \in V_a$ such that
	\begin{enumerate}
		\item $\bar x$ and $\bar y_i$ have a common  neighbour in $\Gamma^*_{a,b}(S^\tau)$.
		\item The first column of $\bar y_i$ is $A$-conjugate to $s_i.$
		\item $\bar x$ and $\bar y_i$ differ only for 2-columns.
	\end{enumerate} 
\end{lemma}
\begin{proof}Since $b\geq 2$ and $d(S)=2,$ there exists $i$ such that $s_i=(1,\dots,1)$. Hence we may assume $s_1=(1,\dots,1)$.
Let $z\in V_b$ be adjacent to $\bar x$ in $\Gamma^*_{a,b}(S^\tau).$ We identify $z$ with a matrix $(t_1,\dots,t_\tau)$ where $t_j\in S^b$ for every $j.$ The columns of the matrix
	$$E:=\begin{pmatrix}s_1&\dots&s_1&\dots&s_u&\dots&s_u\\
	t_1&\dots&t_{\tau_1}&\dots&t_{\tau-\tau_u+1}&\dots&t_\tau
	\end{pmatrix}$$
	are a set of representatives of the $A$-orbits on the generating $d$-tuples of $S.$ Since $s_1=(1,\dots,1),$ $t_1$ must be a generating $b$-tuple of $S$, so $(s_i,t_1)$ (being a generating $d$-tuple of $S$) is $A$-conjugate to the $j$-th column of  $E$ for some $\tau_1<j\leq \tau.$ This means that the $j$-th column of $E$ is
	$$\begin{pmatrix}
	s_i\\t_1^\alpha
	\end{pmatrix}$$ for some $\alpha\in A.$ It follows that
	if we replace the first column of $E$ with
	$$\begin{pmatrix}
	s_i^{\alpha^{-1}}\\t_1
	\end{pmatrix}$$  and the $j$-th column with 
	$$\begin{pmatrix}
	s_1\\t_1^\alpha
	\end{pmatrix}$$ we get a matrix
	$E^*$, corresponding to an edge in $\Gamma^*_{a,b}(S^\tau)$ between $z$ and an element $\bar y_i$, obtained from $\bar x$ by replacing the first column with $s_i^{\alpha^{-1}}$ and the $j$th-column with $s_1$.
\end{proof}

\begin{thm}\label{connesso} Let $\tau=\tau_{a+b}(S).$  Then the graph $\Gamma^*_{a,b}(S^\tau)$ is connected.
\end{thm}
\begin{proof}Clearly the star $\Gamma^*_{0,b}(S^\tau)$ is connected and
	$\Gamma^*_{1,1}(S^\tau)$ is connected by \cite[Theorem 3.1]{pwesim}, so we may assume $0<a\leq b$ and $(a,b)\neq (1,1).$ In particular $b\geq 2.$ 
Let $W_a$ be the set of the elements of $V_a$ which belong to the connected 
	component of $\Gamma^*_{a,b}(S^\tau)$ which contains the vertex $\bar x.$
	The set $W_a$ is a block for the action of $K$ on $V_a$. In particular the setwise stabilizer $H$ of $W_a$ in $K$ contains the point stabilizer $C=C_K(\bar x)=\prod_{1\leq i \leq u}C_{A}(s_i)\wr \mathrm{Sym}(\tau_i).$ We identify  $K$ with $A \wr \mathrm{Sym}(\tau)$: in particular every element $k\in K$ can be written in the form $k=(a_1,\dots,a_\tau)\sigma$
	with $a_i \in A$ and $\sigma \in \perm(\tau)$ and the map $\pi: k \mapsto \sigma$ is a group homomorphism from $K$ to $\perm(\tau)$.
%
%
%
	
Since $C\leq H$, we have  $C^\pi=\prod_{1\leq i \leq u}\mathrm{Sym}(\tau_i)\leq H^{\pi}.$
	The orbits of $C^\pi$ are $\Omega_1=\{1,\dots,\tau_1\},$ $\Omega_2=\{\tau_1+1,\dots,\tau_1+\tau_2\},\dots,\Omega_u=\{\tau-\tau_u+1,\dots,\tau\}.$ Let $j\in \{2,\dots,u\}$ and choose $\bar y_j$ as in Lemma \ref{neig}.
	It follows from  Corollary \ref{arctrans} that $\bar y_j=\bar x^{k_j}$ for some $k_j \in K$. In particular $\bar y_j \in W_a\cap W_a^{k_j}$
	so, since $W_a$ is a block, $W_a=W_a^{k_j}$ and $k_j\in H.$ Let $\sigma_j=k^\pi_j:$ we have $\sigma_j=(1,i_j)$ with $i_j\in \Omega_j.$ This means that $\perm(\tau)=\langle \sigma_2,\dots,\sigma_u, \perm(\tau_1),\dots,\perm(\tau_u)\rangle\leq 
	\langle k_2,\dots,k_u, C\rangle^\pi\leq H^\pi,$ hence $H^\pi=\perm(\tau).$ We identify $S$ with $\inn (S)\leq A.$
	Let $z \in A$ and consider $k=(z,1,\dots,1)\in K.$ Clearly $\bar x^k=\bar x,$ hence $k\in H.$ But then
	$H$ contains $(z,1,\dots,1)$ for every $z\in A$: being  $H^\pi=\perm(\tau)$, this implies that $H=K.$
	
	
	%
%
%
%

	Now we have  $V_a=\bar x^K=\bar x^H \leq W_a,$ hence $W_a=V_a$ and $\Gamma^*_{a,b}(S^\tau)$ is a connected graph.
\end{proof}

Let $S=\ssl(2,2^p)$ with $p>3$. 
We are going to prove that 
$$\lim_{p\to \infty}\diam(\Gamma^*_{a,b}(S^{\tau_{a+b}(S)}))=\infty,$$ for every pair $a, b$ of positive integer.  Let $q=2^p$.	We have  $|S|=(q^2-1)q$ and $A=\aut(S)= S \rtimes \langle \phi \rangle$ with $\phi$ the Frobenius automorphism.
Note that, since $p\neq 3,$ then $p$ does not divide $|S|$; in particular $\langle \phi \rangle$ is a Sylow $p$-subgroup of $A.$ 	Given $k=(u_1,\dots,u_\tau)\pi \in K \leq A \wr \perm(\tau),$ let $\sigma_k$
be  the number of $i \in \{1,\dots,\tau\}$ with $u_i \notin S$. 

\begin{lemma}\label{numb} Let $k\in K.$\begin{enumerate}
\item If $k\in C$, then
	 $$\sigma_k \leq \begin{cases}6^a\cdot \frac{|S|^b}{p}&\text{if $a\neq 1$}\\3\cdot \frac{|S|^{d-1}}{pq}&\text{otherwise.}\end{cases}$$
\item If $k\in D,$  then
	 $$\sigma_k \leq \begin{cases}6^b\cdot \frac{|S|^a}{p}&\text{if $b\neq 1$}\\3\cdot \frac{|S|^{d-1}}{pq}&\text{otherwise.}\end{cases}$$
	 \end{enumerate}
\end{lemma}
\begin{proof}It suffice to prove (1) (the argument for (2) is the same).	Assume that $s \in S$ has the property that $|C_{A}(s)|$ is divisible by $p.$ By Sylow Theorem, $\phi \in C_{A}(s)^\alpha=
	C_{A}(s^\alpha)$ for some $\alpha \in A.$ It follows that $s^\alpha \in C_S(\phi)=\ssl(2,2)\cong \perm(3).$ In particular, exactly three of the   representatives $\eta_1,\dots,\eta_v$ for the $A$-orbits of
	$S$  satisfy the condition that $p$ divides $|C_{A}(\eta_i)|$. More precisely we may assume:
	\begin{enumerate}
		\item $\eta_1=1;$
		\item $|\eta_2|=2$ and $|C_{A}(\eta_2)|=p\cdot q;$
		\item $|\eta_3|=3$ and $|C_{A}(\eta_3)|=p\cdot (q+1)$.
	\end{enumerate}	
First assume $a\neq 1$. We order the elements $s_1,\dots,s_u \in S^a$ in such a way that $C_A(s_i)\not\leq S$ if and only if $i\leq l.$ If $i\leq l$ and $s_i=(z_1,\dots,z_a),$ then 
we may assume $\{z_1,\dots,z_a\}\subseteq C_S(\phi)\cong \perm(3).$  Hence
		\begin{equation}l\leq 6^{a}.
		\end{equation}
		Moreover if $(s_i,t)$ and $(s_i,t^*)$ are generating $d$-tuples for $S$ which are not $A$-conjugate, then $t$ and $t^*$ belong to different orbits for the action of $C_A(s_i)$ on $S^b,$ so for $i\in \{1,\dots,l\}$
		\begin{equation}
		\tau_i\leq \frac{|S|^b}{|C_A(s_i)|}\leq \frac{|S|^b}{p} \quad \text { and } \quad \sigma_k \leq 6^a\cdot \frac{|S|^b}{p}.
		\end{equation}
The case $a=1$ follows with a similar argument, noticing that if $i\leq l,$ then $s_i\in \{\eta_1,\eta_2,\eta_3\}$ and that $|C_A(\eta_j)|\leq |S|/pq$ for $j\in \{1,2,3\}.$
\end{proof}
\begin{thm} Let $S=\ssl(2,2^p)$ with $p>3$, assume that $a\leq b$ are positive integers and let $\tau=\tau_{a+b}(S).$
	  \begin{enumerate}\item If $a\neq 1$ and $p$ is large enough, then $$\diam(\Gamma^*_{a,b}(S^\tau))\geq  \frac{|S|^{a-1}}{2\cdot 6^{a}}-1.$$ 
	  \item If $a=1$ and $p$ is large enough, then 
	  	$$\diam(\Gamma^*_{1,b}(S^\tau))\geq  \frac{2^p}{6}-1.$$
	  	\end{enumerate}
\end{thm}

\begin{proof}

	By \cite{LS}, the probability $P(S)$ of generating a simple group with 2 elements tends to 1 as $|S|$ tends to infinity. In particular
	if $p$ is large enough, then \begin{equation}\tau\geq \frac{|S|^d}{2|A|}=\frac{|S|^{d-1}}{2p}.\end{equation}

\noindent	Case 1: $a\neq 1.$
First assume $a\neq b$.  Let $(\bar x,\bar y)$ be an edge of 
$\Gamma^*_{a,b}(S^\tau)$ with $\bar x \in S^{a\cdot \tau}$ and $\bar y \in S^{b\cdot \tau}$ and let $C=C_K(\bar x),$ $D=C_K(\bar y).$ We may identify the elements of $V_a$ with the right cosets of $C$ in $K$ and the elements of $V_b$ with the right cosets of $D$ in $K:$ there is an edge between $Cx$ and $Dy$ if and only if $Cx\cap Dy\neq \emptyset.$ Assume in particular that our graph contains the path
$(Cx_1,Dy,Cx_2)$: there exist $c_1,c_2\in C$ and $d_1,d_2\in D$ with
$$c_1x_1=d_1y,\quad c_2x_2=d_2y,$$ hence
$$x_2=c_2^{-1}d_2y=c_2^{-1}d_2d_1^{-1}c_1x_1\in CDC x_1.
$$
More generally if there exists a path of length $2r$ from $Cx_1$ to $Cx_2$ then
$$x_2\in C\underbrace{DC\cdots DC}_{r \mbox{ terms}}x_1.$$
Assume $\diam(\Gamma^*_{a,b}(S^\tau))\geq 2r$. By the previous paragraph 
$$K=C\underbrace{DC\cdots DC}_{r \mbox{ terms}},$$
and in particular there exist $c_0,\dots,c_{r}\in C$ and 
$d_0,\dots,d_{r-1}\in D$  such that
\begin{equation}\label{c1d1}(\phi,\dots,\phi)=c_0d_0\cdots c_{r-1}d_{r-1}c_r.\end{equation}
However, by Lemma \ref{numb}
$$c_0d_0\cdots c_{r-1}d_{r-1}c_r=(w_1,\dots,w_\tau)\rho$$
with $w_i \notin S$ for at most 
$$(r+1)\left(6^{a}\cdot \frac{|S|^b}{p}\right)+r\left(
6^{b}\cdot \frac{|S|^a}{p}\right)\leq \frac{(2r+1) \cdot 6^a\cdot |S|^b}{p}
$$ choices of $i$. Hence
$$\frac{(2r+1) \cdot 6^a\cdot |S|^b}{p}\geq \tau \geq \frac{|S|^{d-1}}{2\cdot p}$$ and
this implies
$$2r+1\geq \frac{|S|^{a-1}}{2\cdot 6^{a}}.$$ 	
	Now assume  $a=b.$ We may choose $\bar x=(x_1,\dots,x_\tau)$ and $\bar y=(y_1,\dots,y_\tau)$  with the property: if $(x_i,y_i)$ and $(y_i,x_i)$ are not $A$-conjugate, then there exist $i^*$ such that $x_{i^*}=y_i$ and $y_{i^*}=x_i.$ Now let $J=\{i\mid (x_i,y_i) {\text { and }} (y_i,x_i) {\text{ are not $A$-conjugate}}\}.$
	We have already noticed that there exists $k=(a_1,\dots,a_\tau)\sigma\in K $ 
	such that $\bar y=\bar x^k$ and $\bar x=\bar y^k.$
	Clearly $k$ can be chosen so that:
	\begin{enumerate}
		\item if $i \in J,$ then $i\sigma=i^*$ and $a_i=1;$
		\item if $i \notin J,$ then $i\sigma=i$ and $(x_i,y_i)^{a_i}=(y_i,x_i).$
	\end{enumerate}
	If $i \notin J,$ then $x_i^{a_i^2}=x_i$ and $y_i^{a_i^2}=y_i,$ hence, since $S=\langle x_i, y_i \rangle,$ we have $a_i^2=1.$
	Since 2 does not divide
	$|A/S|=p$, it must be $a_i \in S.$
	We can conclude that $a_i \in S$ for each $i \in \{1,\dots,\tau\}.$ 
		By \cite[Corollary 5.2]{pwesim}, there exist $r\leq \diam(\Gamma^*_{a,a}(S^\tau))$ and $c_i=(u_{i1},\dots,u_{i\tau})\sigma_i\in C$
	such that
	$(\phi,\dots,\phi)=c_0kc_1\cdots kc_r.$ On the other hand, by Lemma \ref{numb} $$c_0kc_1\cdots kc_r=(w_1,\dots,w_\tau)\rho$$
	with $w_i \notin S$ for at most $$(r+1)\cdot 6^{a}\cdot \frac{|S|^a}{p}
	$$ choices of $i$. Hence
	$$(r+1)\cdot 6^{a}\cdot \frac{|S|^a}{p}\geq \tau \geq \frac{|S|^{d-1}}{2p}$$ and
	this implies
	$$r+1\geq \frac{|S|^{a-1}}{2\cdot 6^{a}}.$$
\noindent Case 2: $a=1.$ The case $a=b=1$ is considered in \cite[Theorem 5.4]{pwesim} so we may assume $a\neq b.$ The argument is similar to the one used in Case 1. Indeed again we can say that $(\phi,\dots,\phi)=c_0d_0\cdots c_{r-1}d_{r-1}c_r$ and so, by Lemma \ref{numb},
$$(r+1)\left(3\cdot \frac{|S|^{d-1}}{pq}\right)+r\left(
6^{d-1}\cdot \frac{|S|}{p}\right)\leq \frac{(2r+1) \cdot 3\cdot |S|^{d-1}}{pq}
$$ choices of $i$. Hence
$$\frac{(2r+1) \cdot 3\cdot |S|^{d-1}}{pq}\geq \tau \geq \frac{|S|^{d-1}}{2\cdot p}$$ and
this implies
$$2r+1\geq \frac{q}{6}.$$ 
\end{proof}
We conclude this section with the following application of Theorem \ref{connesso}.

\begin{thm}
Assume that $G$ is a direct product of finite non-abelian simple groups and let $a, b$ non-negative integers with $a+b\geq d(G).$ Then $\Gamma^*_{a,b}(G)$ is connected.
\end{thm}

\begin{proof}
Assume $G=S_1^{n_1}\times \dots \times S_r^{n_r}$ with $S_1,\dots,S_r$ pairwise non isomorphic nonabelian finite simple groups. We prove our statement by induction on $r.$ Let $d=a+b$ and let $\tau_i=\tau_d(S_i).$ We have that $S_i^{n_i}$ is an epimorphic image of $S_i^{\tau_i},$ so il follows from Theorem \ref{connesso} and Lemma \ref{norsbg} that $\Gamma^*_{a,b}(S_i^{n_i})$ is connected. In particular our statement is true if $r=1.$ Suppose that $r \geq 2$ 
  and let $\Gamma_1=\Gamma^*_{a,b}(S_1^{n_1}\times \dots \times S_{r-1}^{n_{r-1}})$ and  $\Gamma_2=\Gamma^*_{a,b}(S_r^{n_r}).$
By induction $\Gamma_1$ and $\Gamma_2$ are connected graphs. If $a=b,$ then 
$\Gamma_1$ and $\Gamma_2$ are not bipartite, so by \cite[Theorem 1]{lex} we conclude that $\Gamma^*_{a,b}(G)=\Gamma_1\times \Gamma_2$ is connected. Suppone $a\neq b$. In this case $\Gamma_1$ is a connected bipartite graph, with two parts $A\subseteq  (S_1^{n_1}\times\dots\times  S_{r-1}^{n_{r-1}})^a$ and $B\subseteq  (S_1^{n_1}\times\dots\times  S_{r-1}^{n_{r-1}})^b$ and $\Gamma_2$ is a connected bipartite graph, with two parts $C\subseteq  (S_r^{n_r})^a$ and $D\subseteq  (S_r^{n_r})^b.$ It can be easily seen that $\Gamma^*_{a,b}(G)$ can be identified  with the subgraph of $\Gamma_1\times \Gamma_2$ induced by $(A\times C) \cup (B\times D).$ Now let $(x,y)$ be an edge of $\Gamma_1,$ with $x\in A$ and $y\in B.$ The subgraph of
$\Gamma^*_{a,b}(G)$ induced by $(\{x\}\times C)\cup (\{y\}\times D)$ is isomorphic to $\Gamma_2,$ hence is connected. Since this is true for every egde of $\Gamma_1$ and $\Gamma_1$ is connected, we immediately conclude that  $\Gamma^*_{a,b}(G)$ is connected as well.
\end{proof}

\section{Properties of $G$ that can be recognized from the graphs $\Gamma^*_{a,b}(G).$}\label{recog}

In this section we will denote by $\Lambda(G)$ the collection of all the connected components  of the graphs $\Gamma_{a,b}(G)$, for all the possible choices of $a \leq b$ in $\mathbb N.$ However for each of this graph, we do not assume to know from which choice of $a, b$ it arises. In particular $\Lambda(G)$  contains lot of graphs just consisting of only one vertex and with no edge. From these graphs we cannot recover any information, so we may restrict our attention to the collection $\lst$ of all  the connected components  of the graphs $\Gamma^*_{a,b}(G),$ for all $a, b \in \mathbb N.$ We deal with two questions:
\begin{question}\label{domuno} Given a graph $\Gamma \in \lst$, can we determine the integers $a, b$ such that $\Gamma$ is a connected component of $\Gamma^*_{a,b}(G)$ ?
\end{question}
\begin{question}\label{domdue} Which information on $G$ can be deduced from the knowledge of $\lst ?$
\end{question}
We already noticed that a graph $\Gamma \in \lst$ can contain loops: we will denote by $\tilde \Gamma$ the graph obtained from $\Gamma$ by deleting the loops. In this way we produce a new collection $\lti$ of graphs. In this section we will also prove that $\lst$ can be reconstructed  from the knowledge of $\lti$ which means that we do not lose information if we remove all the loops from the graphs (see Corollary \ref{vialoop}).

Since  a bipartite graph has a unique partition (up to switching the two sets) if and only if it is connected, Corollary \ref{c15} tell us  that when $a\ne b$, each connected component  $\Gamma_{a,b}^*(G)$ is a bipartite graph  whose unique partition has two parts, namely $V_a$ and $V_b$, corresponding to elements of $G^a$ and $G^b$ respectively. Note that  if $a+b=d(G)$ and $G$ is not soluble, then  we do not know whether  $\Gamma_{a,b}^*(G)$ is connected.  

\

The generating properties of cyclic groups are quite peculiar and exceptional from many points of view. As a result of this, one is immediately able to decide from the knowledge of $\lst$ whether $G$ is cyclic.

\begin{prop}\label{cic} From the knowledge of either $\lst$ or $\lti$ we may recognize whether $G$ is cyclic, and, when $G$ is cyclic, determine $|G|.$
\end{prop}

\begin{proof} The case $G=1$ is uniquely characterized by the fact that $\lst$, and consequently $\lti$, contains infinitely many copies of the complete graph $K_2$: indeed $\Gamma^*_{0,b}(G)\cong K_2$ for every positive integer $b.$ Now assume that $G$ is a non-trivial cyclic group: only in this case $\lst$ contains two stars (corresponding to $\Gamma_{0,1}^*(G)$ and $\Gamma_{0,2}^*(G)$, respectively) with the property that there is no bipartite graph in $\lst$  with the same number of edges.  
If we imagine removing the loops, then we can still recognize the cyclic groups since we have two situations: either we see only two stars of type $K_{1,1}$, or we still  see  two stars with no bipartite graphs with the same number of edges. In the former case the group is $C_2$ and in the latter one it is any other  cyclic group of order greater than two. Once we know that $G$ is a non-trivial cyclic group, we consider all the stars in $\lst$ sorted  by the increasing number of leaves $u_i$, for $i\geq 0$: they correspond to the graphs $\Gamma_{0,i+1}^*(G).$  Note that  $\Gamma_{1,2}^*(G)$   is the only bipartite graph in $\lst$ with $u_2$ edges and $|G|$ is  the cardinality of the smallest set in the partition of $\Gamma_{1,2}^*(G)$.
\end{proof}

Since we can identify the cyclic groups, from now on we assume, without lost of generality, that  $d(G)\geq 2$. By Lemma \ref{tri} a graph $\Gamma \in \lst$ is either bipartite or contains a 3-cycle. There is a loop around a vertex $x=(x_1,\ldots,x_r)$  if and only if $G=\langle x_1,\dots,x_r\rangle$ and $\Gamma$ is not bipartite. In this case $x$ is adjacent to all other vertices of $\Gamma_{r,r}^*(G)$. We want to analyse in which other cases a vertex of a graph in $\lst$ can have this last property.

\begin{thm}\label{died}
	\label{triang}
	Let $G$ be a non-cyclic finite group. Assume that there exists $\Gamma\in \lst$ containing a 3-cycle and a   vertex $x$ which is adjacent to all the other vertices of $\Gamma$. Then either there is a loop in $\Gamma$ around $x$ or $d(G)=2$ and $G$ is isomorphic either to the Klein group or to the dihedral group  $D_p$, for some odd prime $p$.
 \end{thm}

\begin{proof}
Assume that $x=(x_1,x_2,\ldots,x_r)$. 
Since  $\Gamma$ contains a 3-cycle, it is a connected component of $\Gamma^*_{r,r}(G)$, for $r\geq 1.$
In particular there exists  $y=(y_1,\ldots,y_r)$ such that $G=\langle x_1,x_2,\ldots,x_r, y_1,\ldots,y_r \rangle$. 

First assume $r\geq 2.$ If $x$ has at least  two distinct entries, say  $x_i$ and $x_j$ with $i<j$,  then $$x^*=(x_1,\ldots,x_j,\ldots,x_i,\ldots,x_r)$$ is also a vertex  of $\Gamma,$ since it is adjacent to $y$. Hence $x$ is adjacent to $x^*$  and $G$ is generated by the $r$ elements $x_1,\dots,x_r$: in this case we have a loop around $x$. If $x_1=\dots=x_r$ and $x_1\neq 1,$ then again $x^*=(x_1,1,\ldots,1)$ is adjacent to $y$ and consequently to $x$ and  this implies that $G$ is cyclic. Finally if $x=(1,\dots,1)$, then any tuple of type $(z,1,\ldots,1)$, with $z\in G$, is adjacent to $y$ and consequently to $x$ and again $G$ is cyclic.
	
Now assume $r=1$. As a consequence  $\Gamma$ is a connected component of the generating graph $\Gamma^*_{1,1}(G)$ and $d(G)=2$.	Since $x$ is a non-isolated vertex,  there exists $y$ such that $G=\langle x,y\rangle$. First of all observe that $x$ must have order $2$, otherwise  also $x^{-1}$ would be adjacent to $y$ and, in particular to $x$,  contradicting the fact that $G$ is  $2$-generated.  If $x$ is not the unique involution in $\Gamma$, then $G$ is generated by two involutions and so it is a dihedral group. Otherwise, since the element $x^y$ also generates $G$ with $y$, we have $x=x^y$.  Therefore $x$ belongs to $x \in Z(\langle x, y \rangle)=Z(G)$ and, consequently, $G$ is abelian and  $\Gamma=\Gamma^*_{1,1}(G)$. Since $G$ is not cyclic, we must have that $\langle x \rangle$ has a cyclic complement, say $H$, in $G$ and that $|H|$ is even: but in this case $H$
contains an involution, say $z$, such that $xz$ is a non-isolated involution, contradicting the uniqueness of $x$.

We have so proved that $G$ is isomorphic to the semidirect product of $\langle x\rangle \simeq C_2$ with $\langle t\rangle \simeq C_m$, for some integer $m$.   If a prime $p$ divides $m$, then the element $xt^p$ generates $G$ together with $t$. This implies $x=xt^p$ (otherwise $xt^p$ would be  adjacent to $x$). Hence $t^p=1$ and $n=p$. If $p=2$, then $G\cong C_2\times C_2,$ otherwise $G\cong D_p$.

Note, conversely, that if either $G\cong C_2\times C_2$ or $G\cong D_{p}$, then any involution of $G$ is adjacent to all the other vertices of $\Gamma_{1,1}(G).$
\end{proof}

\begin{cor}\label{loop}  Let $G$ be any non-cyclic group which is  isomorphic neither to $C_2 \times C_2$ nor to $D_p$, for any odd prime $p,$ and let $\Gamma \in \lst.$ There is a loop around a vertex  $x$ of $\Gamma$ if and only if $\Gamma$ contains a 3-cycle and $x$  is adjacent  to all the other vertices of $\Gamma.$ 
\end{cor} 

Due to the exceptional behavior of the loops in $\Gamma \in \Lambda^*(G)$ when $G$ is either $C_2\times C_2$ or $D_{2p},$  it is  useful to be able to determine from $\lst$ whether or not  we are  in one of these cases.

\begin{prop}\label{didd} From the knowledge of either $\lst$ or $\lti$ we may recognize whether $G$ is isomorphic either to the Klein group or to the dihedral group  $D_p$ for some odd prime $p$, and, in that case, determine $|G|.$
\end{prop}

\begin{proof}
It follows from Theorem \ref{died} that $G$  is either the Klein group or  the dihedral group  $D_p$ if and only if every $\Gamma\in \lst$ containing a 3-cycle contains also a vertex adjacent to all the other vertices.
In this case $G$ is the Klein group if and only if $\lst$ contains the complete graph $K_3$. If $K_3$ is not in $\lst,$ then $G\cong D_p$ for some $p.$ In order to determine $p,$ we consider all the stars in $\lst:$ they correspond to $\Gamma_{0,r}^*(G)\cong K_{1,\phi_G(r)},$ with $r\geq2$:
so we may determine $\phi_G(2)=\min_{r\geq 2}\phi_G(r).$ On the other hand
$$\phi_G(2)=4p^2\left(1-\frac{1}{4}\right)\left(1-\frac{1}{p}\right),$$ which is an injective function on $p$, whenever $p\geq 2$. Hence  by the knowledge of $\phi_G(2)$ we recognize $p $ and consequently $|G|$.  
\end{proof}
   
\begin{cor}\label{vialoop}
Let $G$ be a finite group. We may determine $\lst$ from the knowledge of $\lti.$
\end{cor}   
   
\begin{proof}By Propositions \ref{cic} and \ref {didd} we may assume that $G$ is neither cyclic nor dihedral of order $2p.$ But then, by Corollary \ref{loop},  assuming that  we have removed all loops  in advance,  we can easily recognize which vertices have a loop around and put them back.  
   \end{proof}
 
\begin{defn}
Given $\Gamma\in \lst,$ let  $e(\Gamma)$ be  the number of edges, excluding the loops,  $l(\Gamma)$ be the number of loops and set $\nu(\Gamma)=2e(\Gamma)+l(\Gamma)$ if $\Gamma$ contains a 3-cycle, 
$\nu(\Gamma)=e(\Gamma)$ otherwise.
\end{defn}

\begin{prop}\label{ord}
Let $G$ be  a finite group. We may determine $d(G)$ from the knowledge of $\lst.$
\end{prop}  

\begin{proof}By Proposition \ref{cic} we may assume 
$d=d(G)\geq 2.$ We consider all the stars in $\lst$ sorted  by the increasing number of leaves $u_i$, for $i\geq 0$: they corresponds to the graph $\Gamma_{0,d(G)+i}^*(G)$ for $i\in \mathbb N.$ If $\Gamma$ is a connected component of $\Gamma^*_{a,b}(G)$ and $a+b=d,$ then $\nu(\Gamma)\leq u_0.$ On the other hand if $a+b>d,$ then, by Corollary \ref{c15},
$\Gamma=\Gamma^*_{a,b}(G)$ is connected and $\nu(\Gamma)=u_{a+b-d}.$
Let $\Omega$ be the subfamily of $\lst$ consisting of the graphs $\Gamma$ with $\nu(\Gamma)=u_1.$
 Depending on the parity of  $d+1$  we  have the following two situations:
 \begin{enumerate} 
\item  $\Omega$ contains $\Gamma_{0,d+1}^*(G)\cong K_{0,u_1}$, other  bipartite $x=[\frac{d-1}{2}]$  graphs not isomorphic to the star $K_{0,u_1}$ and no graph containing a 3-cycle.
\item  $\Omega$ contains $\Gamma_{0,d+1}^*(G)\cong K_{0,u_1}$, other bipartite $x=[\frac{d-1}{2}]$  graphs  not isomorphic to the star $K_{0,u_1}$ and exactly one graph containing a 3-cycle.
\end{enumerate}
In the  former case $d+1$ is odd and $d=2x.$ In the second case  $d+1$ is even  and $d=2x+1$.
\end{proof}

The following definition is useful to deal with Question \ref{domuno}.

\begin{defn}
	Let $G\neq 1$ be a finite group and let  $\Gamma\in \lst$: we say that
	$\Gamma$ has level $t$ if there exist $a, b$ such that $t=a+b$ and 
	$\Gamma$ is a connected component of $\Gamma^*_{a,b}(G).$
\end{defn}
The following lemma says that this is a good definition.
   
\begin{lemma}\label{uniclev}
	Let $G\neq 1$ be a finite group. If $\Gamma\in \lst$, then the level of $\Gamma$ is uniquely determined.
\end{lemma} 
  \begin{proof}
  Put $d=d(G)$ and assume that $\Gamma_{0,d+i}^*(G)$ has $u_i$ leaves for $i\geq 0$. Let $\Gamma\in \lst.$ If  $\nu(\Gamma) \leq u_0$, then $\Gamma$ has level $d.$ Otherwise   $\nu(\Gamma)= u_i$ for some positive integer $i$ and $\Gamma$ has  level $d+i.$ 
 \end{proof}
 
 \begin{lemma}\label{vavb}
 	Let $G\neq 1$ be a finite group. Let  $a+b>d(G)$ and let  
 	$V_a$ and $V_b$ be the two parts of the bipartite graph $\Gamma^*_{a,b}(G)$ corresponding to elements of $G^a$ and $G^b$ respectively. If $a<b$ then $|V_a|<|V_b|.$
 \end{lemma}
 \begin{proof} 
 	For any $x=(x_1,\ldots,x_a)\in V_a$, there exists a generating $(a+b)$-tuple  $z=(z_1,\ldots,z_{a+b})$ for $G$ such that $x_i=z_i$ for $1\leq i\leq a.$  We have  $y=(z_1,\ldots,z_b)\in V_b$, since its entries generate $G$ together with the $a$-tuple $(z_{b+1},\ldots,z_{a+b})$. We define an injective map $\phi: V_a\to V_b$ by setting $\phi(x)=y.$ Assume by contradiction that $\phi$ is  surjective: it can be easily seen that this implies that every $x\in V_a$ has degree 1 in $\Gamma^*_{a,b}(G)$: by Lemma \ref{degone} this is possible only when $G=1.$  
 	 \end{proof}
 
Now we can give the following answer to Question \ref{domuno}. 
\begin{prop}\label{detab}
Let $G\neq 1$ be a finite group. If $\Gamma\in \lst$ has level at least $d(G)+1,$ then there exists a uniquely determined pair $a\leq b$ such that
$\Gamma \cong \Gamma^*_{a,b}(G).$
\end{prop}
  
\begin{proof}Let $d=d(G)$ and assume that $\Gamma$ has level $r=d+i$ with $i\geq 1$. We easily  recognize the star $\Gamma_{0,r}^*(G)$ and, if $r$ is even, $\Gamma_{r/2,r/2}^*(G),$ which is the unique graph, at that level, containing a 3-cycle. Now we want to sort somehow all the bipartite graphs $\Gamma_{a,b}^*(G)$, with $1\leq a< r/2$ and $b=r-a.$ In this case $\Gamma_{a,b}^*(G)$ is a bipartite graph with the unique partition given by the two sets $V_a$ and $V_b$, and, as we have seen in the previous lemma, $|V_a|<|V_b|.$
We claim that $|V_a|<|V_{a+1}|$ whenever $2a < r-2$. It is enough to
construct $\phi: V_a\to V_{a+1}$ which is injective but not surjective.
 For any  $x=(x_1,\ldots,x_a)\in V_a$,  there exists $y=(y_1,\ldots,y_b)\in V_b$ such that  $G=\langle x_1,\ldots,x_a,y_1,\ldots,y_b \rangle$.   Therefore the $(a+1)$-tuple $(x_1,\ldots,x_a,y_1)$ is obviously an element of $V_{a+1}$, since it generates $G$ with the tuple $(y_2,\ldots,y_b)$. We set $\phi(x)=(x_1,\ldots,x_a,y_1).$ The map $\phi$ defined in this way is clearly injective. 	
 As in the proof of the previous lemma, it can be easily seen that $\phi$ is not surjective.
\end{proof}

The remaining part of this section is devoted to collect answers to Question \ref{domdue}.

\begin{prop}\label{dord}
	Let $G$ be  a finite group. We may determine $|G|$ from the knowledge of $\lst.$
\end{prop}
\begin{proof}By Proposition \ref{ord} we may determine $d=d(G).$ Moreover by Lemma \ref{uniclev} and Proposition
 \ref{detab}, we may identify the graph  $\Gamma=\Gamma_{1,d}^*(G)$, which is a bipartite graph with a unique partition in two parts. The two parts are $V_1$ and $V_d$. By Lemma \ref{vavb} $|G|=|V_1|<|V_d|.$ 
\end{proof}

 An immediate consequence of the results in this section is:

\begin{thm}\label{pzg}
	Let $G$ be  a finite group. We may determine $P_G(s)$ from the knowledge of $\lst.$
\end{thm}

\begin{cor}\label{corpzg}
Let $G$ be  a finite group. From the knowledge of $\lst$ we may determine
 whether $G$ is soluble, whether $G$ is supersoluble and, for every prime power $n,$ the number of maximal subgroups of $G$ of index $n$.
\end{cor}
\begin{proof}
If we know $\lst$, then we know $P_G(s)$ and so we may deduce whether $G$ is soluble (\cite[Theorem 5]{arac}), whether $G$ is supersoluble (\cite[Corollary 6]{arac}) and for every prime power $n,$ the number of maximal subgroups of $G$ of index $n$ (\cite[Corollary 18]{arac}).
\end{proof}

Although several properties of $G$ can be recognized by the knowledge of the coefficients of the Dirichlet polynomial $P_G(s)$, this is not always the case. For example we cannot deduce from $P_G(s)$ whether $G$ is nilpotent. Consider for example $G_1=C_6\times C_3$ and 
$G_2=\perm(3)\times C_3.$ It turns out that
$$P_{G_1}(s)=P_{G_2}(s)=\left(1-\frac 1 {2^s}\right)\left(1-\frac 1 {3^s} \right)\left(1-\frac 3 {3^s}\right).$$ We want to show that nevertheless $\lst$ encodes enough information to decide whether $G$ is nilpotent. Before proving this result, we need an auxiliary lemma. 

\begin{lemma}\label{unfat}Let $\alpha=(a_1,\dots,a_r),$ $\beta=(b_1,\dots,b_s)$ be two sequences of prime integers, with $a_1\leq \dots \leq a_r$ and
	$b_1\leq \dots \leq b_s.$
	If
	$$\prod_i\left(1-\frac{1}{a_i}\right)=\prod_j\left(1-\frac{1}{b_j}\right),$$ then $\alpha=\beta.$
\end{lemma}
\begin{proof}
	By induction on $r+s.$ We have
	\begin{equation}\label{rs}
	\prod_i a_i\prod_j(b_j-1)=\prod_i (a_i-1)\prod_jb_j.
	\end{equation}
	Let $p=\max\{a_1,\dots,a_r,b_1,\dots,b_s\}$, $r^*=\max\{i \mid a_i\neq p\}$, $s^*=\max\{j \mid b_j\neq p\}$. Since $p$ does not divides $a_i-1,$ $b_j-1$, divides $a_i$ if and only if $i>r^*$ and divides $b_j$ if and only if $j>s^*,$ we deduce that $r-r^*$ is the multiplicity of $p$ in the left term of (\ref{rs}) and $s-s^*$ is the multiplicity of $p$ in the right term of (\ref{rs}). In particular $r-r^*=s-s^*$ and $a_{r^*+1}=\dots=a_r=b_{s^*+1}=\dots=b_s=p.$ But then
	$$\prod_{i\leq r^*}\left(1-\frac{1}{a_i}\right)=\prod_{j\leq s^*}\left(1-\frac{1}{b_j}\right),$$ and we conclude by induction.
\end{proof}

\begin{thm}\label{nilpo}
	Let $G$ be  a finite nilpotent group. 
	If $H$ is a finite group and  ${\Lambda\!^*\!(H)}=\lst$, then $H$ is nilpotent.
\end{thm}

\begin{proof}
Let $G$ be a finite nilpotent group. For every $p\in \pi(G)$ let $d_p=d(P)$ where $P$ is a Sylow $p$-subgroup of $G.$ For every nonnegative integer $\delta$ consider the Dirichlet polynomials
$$Q_{p,\delta}(s)=\prod_{0\leq i\leq \delta-1}\left(1-\frac{p^i}{p^s}\right),\quad \tilde Q_{p,\delta}(s)=\prod_{1\leq i\leq \delta}\left(1-\frac{p^i}{p^s}\right).$$ 
We have 
$$P_G(s)=\prod_{p\in \pi(G)}Q_{p,d_p}(s).$$
Since  ${\Lambda\!^*\!(H)}=\lst$, it follows from Theorem \ref{pzg} and Corollary \ref{corpzg} that $P_H(s)=P_G(s)$ and that $H$ is a finite supersoluble group with $d(H)=d(G)=d.$ By Lemma \ref{uniclev} and Proposition \ref{detab} in ${\Lambda\!^*\!(H)}=\lst$ we may uniquely identify the graph $\Delta=\Gamma^*_{1,d}(G)=\Gamma^*_{1,d}(H)$: it is a bipartite graph whose partition has two parts $V_1$ and $V_d$  such that $|V_1|=|G|=|H|.$
We are going to use the knowledge of the degrees of the vertices of $V_d$ to deduce that $H$ must be nilpotent.

Since $H$ is supersoluble,  $\overline H =  H/\frat(H)$ can be written in the form
$$\overline H =  H/\frat(H) \cong (W_1^{r_1}\times \dots \times W_t^{r_t})\rtimes\ X,$$
where $X$ is abelian, $|W_i|=p_i$ for a suitable prime $p_i$ and each $W_i$ is non-central. For every $p\in \pi(X),$ let $\delta_p=d(Q),$ where $Q$ is a Sylow $p$-subgroup of $H.$ By \cite[Satz 2]{gaeu}, we have
$$P_H(s)=\prod_{p\in \pi(X)}Q_{p,\delta_p}(s)\prod_{1\leq i\leq t} \tilde Q_{p_i,r_i}(s).$$ Let $\pi=\{p_1,\dots,p_t\}.$ Since $P_G(s)=P_H(s),$ by \cite[Lemma 16]{arac} we deduce that the primes $p_1,\dots,p_t$ are pairwise distinct, $d_{p_i}=r_i+1$ and $\delta_{p_i}=1$ for $1\leq i\leq t.$ Moreover $d_p=\delta_p$ if $p\in \pi(G)\setminus \pi.$ 

If $\omega=(g_1,\dots,g_d)\in G^d$ corresponds to a non-isolated vertex of $\Delta,$  then the degree of $\omega$ in $\Delta$ is
$\delta_\omega=|G|P_G(S,1),$ with $S=\langle g_1,\dots,g_d\rangle$ (here we denote by $P_G(S,1)$ the probability than a randomly chosen element of $G$ generates $G$ together with $S$). Notice that $P_G(S,1)=P_G(S\frat(G),1)=P_{G/S\frat(G)}(1)$ so there exists a subset $\pi_\omega$ of $\pi(G)$ such that 
\begin{equation}\label{161}\delta_\omega=|G|\prod_{p\in \pi_\omega}\left (1-\frac{1}{p}\right)=|V_1|\prod_{p\in \pi_\omega}\left (1-\frac{1}{p}\right).
\end{equation}
In order to conclude that $H$ is nilpotent, it suffices to prove that $\pi=\{p_1,\dots,p_t\}=\emptyset.$ Assume, by contradiction, $\pi\neq \emptyset,$ and let $q=p_1.$ We have $X=Y\times Q,$ where $Q$, the Sylow $q$-subgroup of $X,$ is cyclic. 
Let $K$ be a subgroup of $H$ such that
$$\overline K=K/\frat(H)=(W_1^{r_1-1}\times \dots \times W_t^{r_t})\rtimes\ Y.$$ It can be easily seen that $d(\overline K)\leq d(\overline H)=d.$ So there exists $(h_1,\dots,h_d)\in H^d$ such that $K=\langle h_1,\dots,h_d\rangle \frat H.$ Let $\alpha=(h_1,\dots,h_d):$ we have 
$$\delta_\alpha=|H|P_H(\langle h_1,\dots,h_d\rangle,1)=|V_1|P_{\overline H}(\overline K,1)=|V_1|\left (1-\frac{1}{p_1}\right)^2.$$
We deduce from (\ref{161}) that there exists $\pi \subseteq \pi(G)$ such that
$$\prod_{p\in \pi}\left (1-\frac{1}{p}\right)=\left (1-\frac{1}{p_1}\right)^2,$$
in contradiction with Lemma \ref{unfat}.
\end{proof}

Another piece of information that we cannot recover from the knowledge of $|G|$ and $P_G(s)$ is the order of $\frat(G).$ For example consider $$G_1=\langle x, y \mid x^5=1, y^4=1, x^y=x^2\rangle$$ and $$G_2=\langle x, y \mid x^5=1, y^4=1, x^y=x^4\rangle.$$ We have  $|G_1|=|G_2|=20$ and
$$P_{G_1}(s)=P_{G_2}(s)=\left(1-\frac{1}{2^s}\right)\left(1-\frac{5}{5^s}\right)$$
however $\frat(G_1)=1$ and $\frat(G_2)=\langle x^2\rangle.$ This motivates the following proposition.

\begin{prop}\label{frattini}	Let $G$ be  a finite group. We may determine $|\frat(G)|$ from the knowledge of $\lst.$
\end{prop}

\begin{proof}
	Since $G$ is finite, there exists $\delta\in \mathbb N,$ such that $d(H)\leq \delta$ for every $H\leq G.$ Let $t\geq \delta$ and consider the graph $\Gamma^*_{1,t}(G)$ (we may identify this graph by Lemma \ref{uniclev} and Proposition 
	\ref{detab}). In $V_t$ there are some vertices (the ones corresponding to the generating $t$-uples of $G$) that are adjacent to all the vertices in $V_1.$ We remove these vertices and the edges starting from them. We obtain a new bipartite graph in which some vertices of $V_1$ are isolated: let $\Omega_t$ be the set of these vertices. Notice that $(g)\in \Omega_t$ if and only $\langle g,x_1,\dots,x_t\rangle \neq G$ whenever $\langle x_1,\dots,x_t\rangle \neq G.$ Since $d(H)\leq t$ for every $H\leq G$, we deduce that $(g)\in \Omega_t$  if and only $\langle g, H \rangle \neq G$ whenever $H \neq G.$ In other words $(g)\in \Omega_t$ if and only if $g\in \frat(G).$ We conclude that we may determine $n=|\frat(G)|$ from the fact that $|\Omega_t|=n$ if $t$ is sufficiently large.
\end{proof}

\begin{cor}
	Let $G$ be  a finite non-abelian simple group. 
	If $H$ is finite group and  ${\Lambda\!^*\!(H)}=\lst$, then $H\cong G.$
\end{cor}

\begin{proof}
By Theorem \ref{pzg} $P_G(s)=P_H(s),$ hence $H/\frat(H)\cong G$ by \cite [Theorem 1]{pat}. Moreover, by the previous proposition, $|\frat(H)|=|\frat(G)|=1,$ hence $H\cong G.$
\end{proof}

\begin{lemma} Assume that $\lst$ is known and let $a,b$ be a pair of non-negative integers. If either $a+b>d(G)$ or $a+b=d(G)$ and $G$ is soluble, then we may determine the graph $\Gamma_{a,b}(G/\frat(G)).$
\end{lemma}

\begin{proof}
	Let $f=|\frat(G)|.$ Under our assumptions we know that $\Gamma_{a,b}^*(G)$ is connected. First assume $a\neq b:$ $\Gamma_{a,b}^*(G)$ is a bipartite graph with $|V_a|+|V_b|$ vertices, while
	$\Gamma_{a,b}(G)$ has $|G|^a+|G|^b$ vertices.
	In particular $\Gamma_{a,b}(G)$ is uniquely determined from  $\Gamma_{a,b}^*(G)$: it suffices to add $|G|^a-|V_a|+|G|^b-|V_b|$ isolated vertices. Similarly, if $a=b,$ then $\Gamma_{a,b}(G)$ can be obtained  from  $\Gamma_{a,b}^*(G)$ by adding $|G|^a-|V|$ isolated vertices to the set $V$ of the vertices of  $\Gamma^{*}_{a,b}(G).$ In both  cases we note
 that if $\langle x_1,\dots,x_a,y_1,\dots,y_b\rangle=G$, 
	then $\langle x_1\alpha_1,\dots,x_a\alpha_a,y_1\beta_1,\dots,y_b\beta_b\rangle=G$ 
	for every $\alpha_i, \beta_j \in \frat(G).$ We may consider the following equivalent relations in $\Gamma_{a,b}(G):$ $\omega_1 \sim_1 \omega_2$ if and only if $\omega_1$ and $\omega_2$ have the same neighbourhood  in the graph;  $\omega_1=(x_1,\dots,x_\gamma) \sim_2 
	(y_1,\dots,y_\gamma)$, 
	with $\gamma\in \{a,b\}$, if and only if for any $j$ there exists $f_j\in \frat(G)$ with $y_j=x_jf_j.$ For every vertex $x=(x_1,\dots,x_\gamma)$ of $\Gamma_{a,b}(G)$, the equivalence class $\Omega_x=[x]_{\sim_1}$ is the disjoint union of $|\Omega_x|/f^\gamma$ $\sim_2$-equivalence classes: we obtain $\Gamma_{a,b}(G/\frat(G))$ from $\Gamma_{a,b}(G),$ by deleting from every equivalence class $\Omega_x$ precisely $|\Omega_x|(1-1/f^\gamma)$ vertices.
\end{proof}

By the previous results, at least in the case of finite soluble groups, the knowledge of $\lst$  is equivalent to the knowledge of ${\Lambda\!^*\!(G/\frat(G)})$ and $|\frat G|.$ 

\

From what we proved in this section, a question naturally arises:
\begin{question}
Assume that $G$ is a (soluble) group with $\frat(G)=1.$ Is $G$ uniquely determined from $\lst$?
\end{question}

The answer is negative. Indeed, consider the following example. Let $C_1=\langle x_1\rangle$ and $C_2=\langle x_2\rangle$ be two cyclic groups of order $5$ and let $V_1 = \langle a_1, b_1 \rangle$,
$V_2 = \langle a_2, b_2 \rangle$ be two vector space over the field with 11 elements. We define an action of $C_1$ on $V_1$ in which $x_1$ takes $a_1$ to $3a_1$ and $b_1$ to $4b_1$, and an action of $C_2$ on $V_2$ in which  $x_2$ takes $a_2$ to $3a_2$ and $b_2$ to $5b_2$. The semidirect products $G_1=V_1\rtimes C_1$ and $G_2=V_2\rtimes C_2$ are both of order $605$.  It is easy to see that $G_{1} \not\cong 
G_{2}$, since  every element of $C_1$ has determinant $1$ while this is not true for $C_2$. For $j=1,2$ let $W_{1,j}=\langle a_j \rangle,$ $W_{2,j}=\langle b_j \rangle$ and let $\pi_{i,j}$ be the projection $G_j\to G_j/W_{i,j}.$
We now construct a bijection
$\tau: G_1\to G_2$  in the following way:
\begin{itemize}
	\item we set $\tau((\alpha a_1+\beta b_1)x_1^\gamma)=(\alpha a_2+\beta b_2)x_2^\gamma$ if $\gamma= 0,1 \mod 5$;
	\item  let $g=(\alpha a_1+\beta b_1)x_1^\gamma$ with $\gamma \neq 0 \mod 5.$ There exist
	$\alpha^*, \beta^*$ (depending on $\alpha,\beta,\gamma)$ such that
	$g=((\alpha^* a_1+\beta^* b_1)x_1)^\gamma.$ We set
	$g^\tau=((\alpha^* a_2+\beta^* b_2)x_2)^\gamma.$
\end{itemize}
For $i\in \{ 1,2\},$  $\tau$ induces a bijection $\tau_i: G_1/W_{i,1}\to 
G_2/W_{i,2}.$ We have
\begin{equation}\label{biett}
\langle g^{\tau\pi_{i,2}} \rangle = \langle g^{\pi_{i,1}} \rangle^{\tau_i}.
\end{equation}

We claim that 
$\langle g_1,\dots,g_d\rangle=G_1$ if and only if  $\langle g_1^\tau,\dots,g_d^\tau\rangle=G_2.$ Clearly this claim implies that $\tau$ induces
a graph isomorphism between $\Gamma_{a,b}(G_1)$ and $\Gamma_{a,b}(G_2)$ for every pair $a, b$ of non-negative integers. To prove the claim notice that $\langle y_1,\dots,y_d\rangle=G_j$ if and only if
$\langle y_1^{\pi_{i,j}},\dots,y_d^{\pi_{i,j}}\rangle=G_j/W_{i,j}$ for $i\in \{ 1,2\}$ and that  $\langle y_1^{\pi_{i,j}},\dots,y_d^{\pi_{i,j}}\rangle=G_j/W_{i,j}$ if and only if there exist $k_1, k_2$ with $\langle y_{k_1}^{\pi_{i,j}}\rangle\neq \langle y_{k_2}^{\pi_{i,j}}\rangle.$ So assume $\langle g_1,\dots,g_d\rangle=G_1$ 
and fix $i\in \{1,2\}.$ There exist $k_1, k_2$ with $\langle g_{k_1}^{\pi_{i,1}}\rangle\neq \langle g_{k_2}^{\pi_{i,1}}\rangle.$
It follows from (\ref{biett}), that
$$\langle g_{k_1}^{\tau\pi_{i,2}} \rangle = \langle g_{k_1}^{\pi_{i,1}} \rangle^{\tau_i}\neq \langle g_{k_2}^{\pi_{i,1}} \rangle^{\tau_i}=\langle g_{k_2}^{\tau\pi_{i,2}} \rangle,$$
and so we conclude $\langle g_1^\tau,\dots,g_d^\tau\rangle=G_2.$

\

We conclude by observing that most of the arguments in this section 
use only  part of the information given by the family $\lst.$ In particular it seems a natural question to ask whether a smaller family of graphs can efficiently encode the generating property of $G.$ In some crucial steps of the proofs of our results
(for example in the proof of Theorem \ref{nilpo} and Proposition \ref{frattini}) a decisive role is played by the graphs $\Gamma^{*}_{1,t}(G)$. 
So a good candidate to consider seems to be the family $\lstt$ of the connected components of the graphs $\Gamma^{*}_{1,t}(G)$ for $t\in \mathbb N.$
We assume $\lstt=\{\Delta_k\}_{k\in \mathbb N}$, where the graphs are enumerated in such a way that  $\nu(\Delta_k)\leq \nu(\Delta_{k+1})$ for every $k\in \mathbb N.$

\begin{thm}\label{lstt}Assume that the family $\lstt$ is known. We may determine $|G|,$ $d(G)$, $P_G(s)$ and $|\frat(G)|.$ Moreover we may recognize whether or not $G$ is soluble, supersoluble, nilpotent.
\end{thm}

\begin{proof}
If $G$ is cyclic, then $\Delta_0=\Gamma^*_{1,0}(G)$ is a non-trivial connected graph 
containing a vertex of degree 1, while, by Lemma \ref{degone}, if $G$ is not cyclic none of the graphs
$\{\Delta_k\}_{k\in \mathbb N}$ can contain a vertex of degree 1. So we may recognize from $\lstt$ whether $G$ is cyclic. Therefore, from now on we will assume that  $G$ is not cyclic.

Let $d=d(G).$ There exists $\tau\in \mathbb N$ such that $\Delta_0,\dots,\Delta_\tau$ are the connected components of $\Gamma^*_{1,d-1}(G).$ By Corollary \ref{c15}, for $k>\tau$ we have $\Delta_k=\Gamma^*_{1,d+k-\tau-1}(G).$ We need to recognize $\tau.$
Notice that if $k>\tau,$ then $\Delta_k$ is a bipartite graph with one of the two parts consisting precisely of $|G|$ vertices and the second part containing a subset of $\phi_G(d+k-\tau-1)$ vertices connected to all the vertices of the first part. We claim that  $\Delta_k$ does not behave in this way whenever $k\leq \tau.$ If $d=2,$ then none of the connected components of $\Gamma_{1,1}(G)$ is bipartite. So we may assume $d\neq 2.$ Assume by contradiction that there exists a connected component of  $\Gamma^*_{1,d-1}(G)$,
say $\Delta,$ which is a bipartite graph with two parts $A$ and $B$ such that
$|A|=G$ and at least one vertex in $B$ is connected to all the vertices in $A.$
Since $(1)$ is an isolated vertex of $\Gamma_{1,d-1}(G)$, it must be $A\subseteq G^{d-1}$ and $B\subseteq G.$ Let $(x)\in B$ be a vertex connected to all the vertices of $A$. Fix $(g_1,\dots,g_{d-1})\in A.$ Since
$$\langle x,g_1,\dots,g_{d-1}\rangle\!=\!\langle x,g_1x,g_2,\dots,g_{d-1}\rangle\!=\!\langle g_1,g_1x,g_2,\dots,g_{d-1}\rangle\!=\!\langle g_1,x,g_2,\dots,g_{d-1}\rangle$$
we have $(x,g_2,\dots,g_{d-1})\in A,$ hence $(x)$ and $(x,g_2,\dots,g_{d-1})$ are adjacent,
but this would imply $G=\langle x,g_2,\dots,g_{d-1}\rangle$, hence $d(G)\leq d-1,$ a contradiction.

Once  $\tau$ has been determined, we have that $|G|$ is the cardinality of the smaller part in the bipartite graph $\Delta_k$, for any choice of $k>\tau.$ Alternatively, we may notice that
$$\lim_{k\to\infty}\frac{\nu(\Delta_{k+1})}{\nu(\Delta_{k})}=\lim_{k\to\infty}\frac{\phi_G(d+k-\tau+1)}{\phi_G(d+k-\tau)}=|G|.$$
We can also determine $d(G),$ since $\nu(\Delta_k)=\phi_G(d+k-\tau) \sim |G|^{d+k-\tau}$ if $k$ is large enough and so
$$d=\lim_{k\to \infty}\log_{|G|}(\nu(\Delta_k))-k+\tau.$$
But now we know $P_G(k)$ for every positive integer $k\neq d(G)$ and this is enough to determine the Dirichlet polynomial $P_G(s).$ In particular we may recognize whether $G$ is soluble, supersoluble, nilpotent (for this we repeat the argument in Theorem \ref{nilpo}).  Moreover we may determine $|\frat(G)|$ (same proof as Proposition \ref{frattini}).
\end{proof}

\section{Generalizing some definitions and results from \cite{ccl}}

The following equivalence relation  $\equiv_{\remm}$ was introduced in \cite[Section 2]{ccl}: two elements
are equivalent if each can be substituted for the other in any
generating set for $G$. By \cite[Proposition 2.2]{ccl},  $x\equiv_\remm y$ if and only if
$x$ and $y$ lie in exactly the same maximal subgroups of $G$.
We then refine this to 
a sequence $\equiv_\remm^{(r)}$ of equivalence relations by saying that,
for any positive integer $r$,  $x\equiv_\remm^{(r)}y$ if and only if
\[(\forall z_1,\ldots,z_{r-1}\in G)\quad ((\langle x,z_1,\ldots,z_{r-1}\rangle=G)
\Leftrightarrow(\langle y,z_1,\ldots,z_{r-1}\rangle = G)).\]
Notice that  $x\equiv_\remm^{(r)}y$ if and only if $(x)$ and $(y)$ have the same neighbours in the graph $\Gamma_{1,r-1}(G)$: in particular  $\Gamma_{1,r-1}(G)$ determines the number of classes for the equivalence relation $\equiv_\remm^{(r)}$ and the sizes of these classes.
The relations $\equiv_\remm^{(r)}$ become finer as $r$ increases. 
We define
a  group invariant
$\psi(G)$ to be the value of $r$ at which the relations $\equiv_{\remm}^{(r)}$
stabilise to $\equiv_{\remm}$.
If $G$ is soluble then $\psi(G) \in \{d(G), d(G) +1\}$ (see \cite[Corollary 2.12]{ccl}). Furthermore, in general
$d(G) \leq \psi(G) \leq d(G) + 5$ (see \cite[Corollary 2.13]{ccl}),  however no example is known of a finite
group $G$ for which $\psi(G) > d(G) + 1$. 	For $r\geq \psi(G)$, we have that $(x)$ and $(y)$  have the same neighbours in the graph $\Gamma_{1,r-1}(G)$ if and only if $x\equiv_\remm y$. In particular from the knowledge of the family of graphs $\{\Gamma_{1,r-1}(G)\}_{r\in \mathbb N}$ we may  determine the precise value of $\psi(G).$

\

 Given a subset
 $X$ of 
 a finite group $G,$ we will denote by $d_X(G)$ the smallest cardinality
 of a set of elements of $G$ generating $G$ together with the elements 
 of $X.$ In \cite[Definition 2.15]{ccl} the following notion is also introduced: a finite group $G$ is efficiently generated if for all $x\in G$, 
$d_{\{x\}}(G)=d(G)$ implies that $x\in \frat(G).$ 
\begin{prop}Assume that the family $\lstt=\{\Gamma^*_{1,r-1}(G)\}_{r\in \mathbb N}$ is known. We may deduce whether $G$ is or not efficiently  generated.
\end{prop} 
\begin{proof}First, by Theorem \ref{lstt}, we may determine $d(G)$, $|G|$ and $|\frat(G)|.$ Moreover, inside the family $\lstt$, we may identify the  connected components of \linebreak $\Gamma^*_{1,d(G)-1}(G)$ and consequently 
we may count how many of the vertices of \linebreak $\Gamma_{1,d(G)-1}(G)$ corresponding to 1-tuples are isolated. Let $\omega$ be the number of these vertices: $G$ is efficiently generated if and only $|\frat(G)|=\omega.$ 
\end{proof}

\begin{cor}\label{corpsi}Assume that  the family $\lstt$ is known. If $G$ is soluble, then we may determine $\psi(G).$
	\end{cor}
\begin{proof}
Assume that $G$ is a finite soluble group. By \cite[Corollary 2.20]{ccl}, $\psi(G)  = d(G)$ if $G$ is efficiently generated, $\psi(G)=d(G)+1$ otherwise, so the conclusion follows immediately from the previous proposition.
\end{proof}

Generalizing a definition given in \cite{ccl} for 2-generator groups, we say that a finite $G$ has non-zero spread if $(g)$ is not isolated in the graph $\Gamma_{1,d(G)-1}(G)$ for every $g\neq 1.$ Moreover  we define an equivalence
relation $\equiv_\rg$ on the elements of $G$ by the rule $x\equiv_\rg y$ if $(x)$ and $(y)$
have the same set of neighbours in the graph $\Gamma_{1,d(G)-1}(G)$.
The following statements generalize \cite[Proposition 4.5]{ccl} and 
\cite[Theorem 4.6]{ccl} and can be easily proved.

\begin{prop}
	Let $G$ be a finite group. Then 
	the relations  $\equiv_\rg$ and $\equiv_\remm^{(d)}$ on
	$G$ coincide; hence
	$\equiv_\remm$ is a refinement of $\equiv_\rg$, and is equal to
	$\equiv_\rg$ if and only if $\psi(G) \leq d$. 
\end{prop}

\begin{thm}\label{thm:spread_phi}
	Let $G$ be a finite group with $d(G) = d$. 
	\begin{enumerate}
		\item[(1)] $G$ has non-zero spread
		if and only if $G$ is efficiently generated and has trivial
		Frattini subgroup. 
		\item[(2)] If $G$ is soluble and has non-zero spread, then $\psi(G) =
		d$.
	\end{enumerate}
\end{thm}

Assume that $G$ is a finite group with non-zero spread and let $d=d(G).$
If $N$ is a non-trivial  normal subgroup of $G$, then 
$d(G/N)<d$ (otherwise we would have $d_{\{y\}}(G)=d$ for every $y \in
N$).  So $G$ has the following property:  
$$(\star) \quad \text {every proper quotient can be generated by $d-1$ elements, but $G$ cannot}.$$
When $d(G)=2$, groups with non-zero spread are also called $\frac 3 2$-generated.
In \cite{bgk}, Breuer, Guralnick and Kantor make the following remarkable conjecture: a finite group is  $\frac 3 2$-generated if and only if every proper quotient is cyclic.
In our terminology we could propose  the following more general  conjecture:
\begin{con}\label{cnzs}
A finite group $G$ has non-zero spread if and only if $G$ satisfies  the property $(\star).$
\end{con}

The groups with this property $(*)$ have been studied in \cite{dvl}. By \cite[Theorem 1.4 and Theorem 2.7]{dvl}, there exists a monolithic primitive group $L$ and a positive integer $t$ such that $G\cong L_t$  
and $d(L_{t-1})<d(L_t)$ (setting $L_0=L/\soc(L)).$
This motivates the following question:

\begin{question}\label{qtsp}Let $L$ be a finite monolithic primitive group 
and $t\in \mathbb N.$	Assume that $G\cong L_t$  
	and $d(L_{t-1})<d(L_t).$ Does $G$ have non-zero spread?
\end{question}

The remain part of this section will give an affirmative answer to the previous question.

\

First assume that $N=\soc L$ is nonabelian. If $t=1$ then by \cite[Theorem 1.1]{un} $d=d(G)=d(L)=\max(d(L/N),2)\leq \max(d-1,2)$, hence $d=2$  and Question  \ref{qtsp} has an affirmative answer by Theorem 1 in  \cite{bgh}. 
Suppose $t\neq 1$ (and consequently $d\neq 2$) and let $x=(l,l{n_2},\dots,ln_t),$ 
with $l\in L,$ $n_i\in N$, be a non-identity element of $G=L_t.$	Since $d(L_{t-1})<d,$ there exist $y_1=(l_1,l_1m_{1,2},\dots,l_1m_{1,t-1}),\dots,$ $y_{d-1}=(l_{d-1},l_{d-1}m_{d-1,2},\dots,l_{d-1}m_{d-1,t-1})$ such that
$L_{t-1}=\langle y_1,\dots,y_{d-1}\rangle.$ This is equivalent to saying that
the rows of the matrix
$$A:=\begin{pmatrix}l_1&l_2&\dots&l_{d-1}
\\l_1m_{1,2}&l_2m_{2,2}&\dots&l_{d-1}m_{d-1,2}\\\vdots&\vdots&\cdots&\vdots\\
l_1m_{1,t-1}&l_2m_{2,t-1}&\dots&l_{d-1}m_{d-1,t-1}
\end{pmatrix}
$$
are generating $(d-1)$-tuples of $L$ which belong to distinct orbits with respect to the conjugacy action of $C=C_{\aut L}(L/N)$. Since $x$ is a non-identity element of $G$, there exist $i\in \{2,\dots,t\}$ and $n$ in $N$ such that $l^n\neq ln_i$. Up to reordering, we may assume $i=t.$ Let $\tilde y_1\!=\!(l_1,l_1m_{1,2},\dots,l_1m_{1,t-1},l_1^n),\dots, \tilde y_{d-1}\!=\!(l_{d-1},l_{d-1}m_{d-1,2},\dots,l_{d-1}m_{d-1,t-1},l_{d-1}^n).$
We claim that $L_t=\langle \tilde y_1,\dots,\tilde y_{d-1}, x\rangle.$
This is equivalent to say that
the rows of the matrix
$$\tilde A:=\begin{pmatrix}l_1&l_2&\dots&l_{d-1}&l
\\l_1m_{1,2}&l_2m_{2,2}&\dots&l_{d-1}m_{d-1,2}&ln_2\\\vdots&\vdots&\cdots&\vdots&\vdots\\
l_1m_{1,t-1}&l_2m_{2,t-1}&\dots&l_{d-1}m_{d-1,t-1}&ln_{t-1}\\l_1^n&l_2^n&\dots&l_{d-1}^n&ln_t
\end{pmatrix}
$$
are generating $d$-tuples of $L$ which belong to distinct orbits with respect to the conjugacy action of $C=C_{\aut L}(L/N).$ The way in which $A$ has been constructed ensures that the first $t-1$ rows of $\tilde A$ satisfy the requested properties. We have only to prove that the last row cannot be $C$-conjugate to one of the first $t-1$ rows. Suppose $i\in \{2,\dots,t-1\}:$
since $(l_1,l_2,\dots,l_{d-1})$ and $(l_1m_{1,i},l_2m_{2,i},\dots,l_{d-1}m_{d-1,i})$ are not $C$-conjugate and $n\in C$ we deduce that also $(l_1^n,l_2^n,\dots,l_{d-1}^n,ln_t)$ and $(l_1m_{1,i},l_2m_{2,i},\dots,l_{d-1}m_{d-1,i},ln_i)$ are not $C$-conjugate.
Finally assume by contradiction that there exists $\gamma\in C$ with 
$(l_1^n,\dots,l_{d-1}^n,ln_t)=(l_1,\dots,l_{d-1},l)^\gamma.$ Since $\langle l_1,\dots,l_{d-1}\rangle=L,$ we have $C_C(l_1,\dots,l_{d-1})=1,$ hence $n=\gamma$ and consequently $ln_t=l^n,$ a contradiction. So we have proved that Question \ref{qtsp} has an affirmative answer when $\soc(L)$ is nonabelian.

\

Now assume that $N=\soc L$ is abelian.
We have  $L=N\rtimes H,$ where $H$ is an irreducible subgroup of $\aut(N)$ and $d(H)=d(L/N) \leq d-1.$
As usual, let $F=\End_HN,$ $q=|F|,$ $n=\dim_F(N)$, $m=\dim_F(\der(H,N)).$ 
Let $\delta_1,\dots,\delta_m$ be a basis of $\der(H,N)$ as an $F$-vector space. For each $h\in H$ consider the matrix $A_h\in M_{m\times n}(F)$ defined by setting $$A_h:=\begin{pmatrix}\delta_1(h)\\\vdots\\\delta_m(h)\end{pmatrix}.$$ The following is an immediate consequence of \cite[Proposition 5]{cheb}.
\begin{lemma}\label{che} Suppose that $H=\langle h_1, \dots, h_k\rangle$ and let $u$ be a positive integer.
	Let $w_i=(w_{i,1},\dots,w_{i,u})\in N^t$ with $1\leq i\leq k$ 
	and let $$B_i=\begin{pmatrix}w_{i,1}\\\vdots\\w_{i,u}
	\end{pmatrix}\in M_{t\times n}(F).$$
	The following are equivalent.
	\begin{enumerate}
		\item $L_t = N^t\rtimes H=\langle h_1w_1,\dots,h_kw_k\rangle$;
		\item $\ra\begin{pmatrix}A_{h_1}&\cdots&A_{h_k}\\B_1&\cdots&B_k
		\end{pmatrix}=m+t.$
	\end{enumerate}
	In particular $d(L_t)\leq k$ if and only if $m+t\leq kn.$
\end{lemma}
In our case $d(G)=d(L_t)=d$ but $d(L_{t-1})\leq d-1$, since $L_{t-1}$ is a proper epimorphic image of $L_t:$ by the previous Lemma we must have $m+t-1=(d-1)n$ i.e., $$t=(d-1)n-m+1.$$ Now assume that $x:=h(v_1,\dots,v_t)$ is a non-identity element of $L_t$. Fix $h_1,\dots,h_{d-1}$ such that $H=\langle h_1,\dots,h_{d-1}\rangle.$ There exist $\tilde w_i\in N^{t-1},$ for $1\leq i \leq d-1,$ such that $L_{t-1}=\langle h_1\tilde w_1,\dots,h_{d-1}\tilde w_{d-1}\rangle,$ in other words \begin{equation}\label{invert}\det\begin{pmatrix}
A_{h_1}&\dots&A_{h_{d-1}}\\\tilde B_1&\dots&\tilde B_{d-1}\end{pmatrix}\neq 0.\end{equation} We claim that there exist $u_1,\dots,u_{d-1}\in N$ such that
\begin{equation}\label{genera}L_t=\langle h(v_1,\dots,v_t), h_1(\tilde w_1, u_1),\dots,h_{d-1}(\tilde w_{d-1}, u_{d-1})\rangle.\end{equation}
Set $$\tilde B=\begin{pmatrix}v_1\\\vdots\\v_{t-1}\end{pmatrix}.$$
By Lemma \ref{che}, (\ref{genera}) is equivalent to
\begin{equation}\ra\begin{pmatrix}A_h&A_{h_1}&\cdots&A_{h_{d-1}}\\\tilde B&\tilde B_1&\cdots&\tilde B_{d-1}\\v_t&u_1&\cdots&u_{d-1}\end{pmatrix}=(d-1)n+1=m+t.
\end{equation}
Since $x\neq 1,$ we have $$X:=\begin{pmatrix}A_h\\\tilde B\\v_t\end{pmatrix}\neq 0.$$ In particular at least one column of $X$ is a non-zero element of $M_{m+t,1}(F)$.  Let us write such a  column in the form 
$$Y=\begin{pmatrix}C\\\gamma\end{pmatrix}$$ with $C\in M_{m+t-1,1}(F)$ and $\gamma\in F.$ Let $$Z:=\begin{pmatrix}
A_{h_1}&\dots&A_{h_{d-1}}\\\tilde B_1&\dots&\tilde B_{d-1}\end{pmatrix}.$$
By (\ref{invert}), $C$ is a linear combination of the columns of $Z.$ If $\gamma\neq 0,$ then
$$\det \begin{pmatrix}C&Z\\\gamma&0\end{pmatrix}\neq 0,$$ so we are done if we choose $u_1=\dots=u_{d-1}=0.$ If $\gamma= 0,$ then $C$ is a non-zero matrix, so, denoting by $Z_i$ the $i$-th column of $Z,$ there exists
$(0,\dots,0)\neq (\lambda_1,\dots,\lambda_{(d-1)n})\in F^{(d-1)n}$ such that
$\sum_i\lambda_iZ_i=C.$ Choose $(\alpha_1,\dots,\alpha_{(d-1)n})\in F^{(d-1)n}$ such that $\sum_i\lambda_i\alpha_i\neq 0.$ If we choose $$u_1=(\alpha_1,\dots,\alpha_n), u_2=(\alpha_{n+1},\dots,\alpha_{2n}), \dots,
u_{d-1}=(\alpha_{(d-2)n+1},\dots,\alpha_{(d-1)n}),$$ then
$$\det\begin{pmatrix}C&Z\\0&\begin{matrix} u_1&\dots&u_{d-1}\end{matrix}
\end{pmatrix}\neq 0$$ 
Summarizing we proved:
\begin{prop}\label{ultimo}The answer to Question \ref{qtsp} is affirmative.  As a consequence Conjecture \ref{cnzs} is true.
\end{prop}

\end{document}